\documentclass[11pt]{article}
\usepackage{latexsym}
\usepackage{amsmath}
\usepackage{amsthm}
\usepackage{amssymb}
\usepackage[pdftex]{color}

\ifx\pdftexversion\undefined
  \usepackage[dvips]{graphics}
\else
  \usepackage[pdftex]{graphics}
\fi

\newtheorem{theorem}{Theorem}[section]
\newtheorem{corollary}[theorem]{Corollary}
\newtheorem{lemma}[theorem]{Lemma}
\newtheorem{proposition}[theorem]{Proposition}

\newtheorem{definition}[theorem]{Definition}
\newtheorem{example}[theorem]{Example}

\makeatletter
\@addtoreset{equation}{section}
\makeatother

\newcommand{\hz}{\hat{0}}
\newcommand{\ho}{\hat{1}}
\newcommand{\cv}{\vec{c}}
\newcommand{\cvp}{\vec{c}\:^{\prime}}
\newcommand{\cvc}{\vec{c}\:^{c}}
\newcommand{\dv}{\vec{d}}
\newcommand{\dvp}{\vec{d}\:^{\prime}}
\newcommand{\ev}{{\mathbf e}}
\newcommand{\SSSS}{{\mathfrak S}}
\newcommand{\HH}{\widetilde{H}}

\newcommand{\conv}{\operatorname{conv}}
\newcommand{\Des}{\operatorname{Des}}
\newcommand{\type}{\operatorname{type}}
\newcommand{\sd}{\operatorname{sd}}

\newcommand{\cupdots}{\cup \cdots \cup}
\newcommand{\timesdots}{\times \cdots \times}

\newcommand{\covered}{\prec}
\newcommand{\covering}{\succ}
\newcommand{\meet}{\wedge}
\newcommand{\join}{\vee}

\begin{document}

\title{The topology of restricted partition posets}
\author{Richard Ehrenborg and JiYoon Jung}

\date{}

\maketitle

\begin{abstract}
For each composition $\cv$ we show that
the order complex of the poset of pointed set
partitions~$\Pi^{\bullet}_{\cv}$
is a wedge of spheres
of the same dimension
with the multiplicity
given by the
number of permutations with descent composition $\cv$.
Furthermore, the action of the symmetric group on the
top homology is isomorphic to the Specht module~$S^{B}$
where $B$ is a border strip associated to the composition.
We also study the filter of pointed set partitions
generated by a knapsack integer partition and show
the analogous results on homotopy type and action
on the top homology.
\end{abstract}


\section{Introduction}
\label{section_introduction}

The study of partitions with restrictions on their block sizes began
in the dissertation by Sylvester~\cite{Sylvester}, who studied the poset
$\Pi_{n}^{2}$
of partitions of $\{1,2, \ldots, n\}$
where every block has even size.
He proved that the M\"obius function of this poset
is given by
$\mu(\Pi_{n}^{2} \cup \{\hz\}) = (-1)^{n/2} \cdot E_{n-1}$,
where $E_{n}$ denotes the $n$th Euler number.
Recall that the $n$th Euler number enumerates
alternating permutations, that is, permutations
$\alpha=\alpha_{1} \cdots \alpha_{n}$
in the symmetric group $\SSSS_{n}$
such that $\alpha_{1} < \alpha_{2} > \alpha_{3} < \alpha_{4} > \cdots$.
Stanley~\cite{Stanley_e_s}
generalized this result to the $d$-divisible partition lattice $\Pi_{n}^{d}$,
that is, the collection of partitions
of $\{1,2, \ldots, n\}$
where each block size is divisible by $d$.
He found that the M\"obius function
$\mu( \Pi_{n}^{d} \cup \{\hz\})$ is,
up to the sign~$(-1)^{n/d}$, the number of
permutations in~$\SSSS_{n-1}$ with descent set
$\{d,2d, \ldots,n-d\}$,
in other words, the number of permutations
with descent composition
$(d, \ldots, d, d-1)$.

The enumerative results for the $d$-divisible partition lattice
were extended homologically.
Calderbank, Hanlon and Robinson~\cite{Calderbank_Hanlon_Robinson}
considered the action of the
symmetric group~$\SSSS_{n-1}$ on the
top homology group of the
order complex of $\Pi_{n}^{d} - \{\ho\}$.
They showed this action is the Specht module
on the border strip corresponding the composition
$(d, \ldots, d, d-1)$.
Wachs~\cite{Wachs} showed that
the other reduced homology groups vanish.
She presented an $EL$-labeling
for the $d$-divisible partition lattice
and hence as a corollary obtained that the homotopy type
is a wedge of spheres of dimension $n/d - 2$.
She then gave a more constructive proof of the representation
of the top homology of $\Delta(\Pi^{d}_{n}-\{\hat{1}\})$
by exhibiting an explicit isomorphism.
She identified cycles in
the complex $\Delta(\Pi^{d}_{n}-\{\hat{1}\})$
which are the barycentric subdivision of cubes
and associated with them polytabloids in
the Specht module.

So far we see that the $d$-divisible partition lattice
is closely connected
with permutations having the descent composition
$(d, \ldots, d, d-1)$. We explain this phenomenon in this paper
by introducing pointed partitions.
They are partitions where one block is considered special,
called the pointed block.
We obtain such a partition by removing the element $n$
from its block and making this block the pointed block.
We now extend the family of posets under consideration.
For each composition $\cv = (c_{1}, \ldots, c_{k})$ of $n$
we define a poset $\Pi^{\bullet}_{\cv}$ such that
the M\"obius function $\mu(\Pi^{\bullet}_{\cv} \cup \{\hz\})$ is
the sign~$(-1)^{k}$ times the number
of permutations with descent composition $\cv$.
Furthermore, we show the order complex of $\Pi^{\bullet}_{\cv}-\{\hat{1}\}$
is homotopy equivalent to a wedge of spheres of dimension $k-2$.
Finally, we show the action of the symmetric group
on the top homology group
$\HH_{k-2}(\Delta(\Pi^{\bullet}_{\cv} - \{\ho\}))$
is the Specht module corresponding to
the composition $\cv$.

Our techniques differ from Wachs' method for studying
the $d$-divisible partition lattice~\cite{Wachs}
as we do not obtain an $EL$-labeling of
$\Pi^{\bullet}_{\cv} \cup \{\hz\}$.
Instead we apply
Quillen's Fiber Lemma
and transform the question
into studying a subcomplex $\Delta_{\cv}$
of the complex of ordered
partitions. This subcomplex is in fact
the order complex of a rank-selected Boolean algebra.
Hence $\Delta_{\cv}$ is shellable and its homotopy type
is a wedge of spheres.
Furthermore, we use an equivariant version of
Quillen's Fiber Lemma
to conclude that
the reduced homology groups
$\HH_{k-2}(\Delta(\Pi^{\bullet}_{\cv} - \{\ho\}))$
and
$\HH_{k-2}(\Delta_{\cv})$
are isomorphic as $\SSSS_{n}$-modules.
Finally, to show that
the top homology group
$\HH_{k-2}(\Delta_{\cv})$
is isomorphic to the Specht module~$S^{B}$,
we follow Wachs' footsteps
by giving an explicit isomorphism
between these two $\SSSS_{n}$-modules,
that is, given a polytabloid ${\bf e}_{t}$
in the Specht module~$S^{B}$
we give an explicit cycle $g_{\alpha}$
spanning the homology group $\HH_{k-2}(\Delta_{\cv})$.

Ehrenborg and Readdy~\cite{Ehrenborg_Readdy_I}
introduced the notion of a knapsack partition,
which is an integer partition such that no integer can be written
as a sum of parts of the partition in two different ways.
They considered the filter in the pointed partition lattice
where the generators of the filter have their type
given by a knapsack partition.
They obtained the M\"obius function of this filter
is given by a sum of descent set statistics.
We extend their results topologically by showing that
the associated order complex is a wedge of spheres.
The proof follows the same outline as the previous study
except that we use
discrete Morse theory
to determine the homotopy type
of the associated complexes of
ordered set partitions.
Furthermore we obtain that
the action of the symmetric group on
the top homology is a direct sum of Specht modules.

We end the paper with open questions for future research.

\section{Preliminaries}
\label{section_preliminaries}

For basic notions concerning partially ordered sets (posets),
see Stanley's book~\cite{Stanley_EC_I}.
For topological background,
see Bj\"orner's article~\cite{Bjorner_handbook}
and Kozlov's book~\cite{Kozlov}.
For representation theory
of the symmetric group, see
Sagan~\cite{Sagan}.
Finally, a good reference for all three areas is
Wachs' article~\cite{Wachs_III}.

Let $[n]$ denote the set $\{1,2, \ldots, n\}$
and for $i \leq j$ let~$[i,j]$ denote the interval
$\{i,i+1, \ldots, j\}$.
A {\em pointed set partition} $\pi$
of the set $[n]$
is a pair $(\sigma,Z)$,
where $Z$ is a subset of $[n]$ and
$\sigma = \{B_{1}, B_{2}, \ldots, B_{k}\}$
is a partition of the set difference $[n] - Z$.
We will write the pointed partition $\pi$ as
$$  \pi  =  \{B_{1}, B_{2}, \ldots, B_{k}, \underline{Z}\}  , $$
where we underline the set $Z$ and we write
$1358|4|\underline{267}$ as shorthand
for \{\{1,3,5,8\},\{4\},\underline{\{2,6,7\}}\}.
Let~$\Pi^{\bullet}_{n}$ denote the set of
all pointed set partitions on the set~$[n]$.
The set $\Pi^{\bullet}_{n}$ has a natural poset structure.
The cover relation is given by two relations
\begin{eqnarray*}
  \{B_{1}, B_{2}, \ldots, B_{k}, \underline{Z}\}
 & < &
  \{B_{1} \cup B_{2}, \ldots, B_{k}, \underline{Z}\}  , \\
  \{B_{1}, B_{2}, \ldots, B_{k}, \underline{Z}\}
 & < &
  \{B_{2}, \ldots, B_{k}, \underline{B_{1} \cup Z}\}  .
\end{eqnarray*}

\begin{lemma}
The poset $\Pi^{\bullet}_{n}$ is
the intersection lattice of the hyperplane arrangement
$$    \left\{ \begin{array}{l l}
          x_{i} = x_{j} & 1 \leq i < j \leq n , \\
          x_{i} = 0     & 1 \leq i \leq n .
             \end{array} \right.   $$
\end{lemma}
\begin{proof}
For the pointed partition
$\pi = \{B_{1}, B_{2}, \ldots, B_{k}, \underline{Z}\}$
construct a subspace satisfying
the equalities $x_{j_{1}} = x_{j_{2}}$
if $j_{1}$ and $j_{2}$ belong to the same block
of $\pi$, and let $x_{j} = 0$ if $j$ belongs to the
block $Z$. It is straightforward to see that
this is a bijection and thus proving the lemma.
\end{proof}

As a corollary to this claim we have that
the poset $\Pi^{\bullet}_{n}$ is a lattice.
Moreover, we call the set $Z$
the {\em zero set} or the {\em pointed block.}
The first name is motivated by the fact
that the set $Z$ corresponds to
the variables set to be zero
in an element in the intersection lattice.

The lattice $\Pi^{\bullet}_{n}$ is isomorphic to
the partition lattice $\Pi_{n+1}$ by the bijection
$$ \{B_{1}, \ldots, B_{k}, \underline{Z}\}
     \longmapsto
   \{B_{1}, \ldots, B_{k}, Z \cup \{n+1\}\}  .  $$
However it is to our advantage
to work with pointed set partitions.

For a permutation $\alpha = \alpha_{1} \cdots \alpha_{n}$
in the symmetric group~$\SSSS_{n}$
define its {\em descent set} to be the set
$$   \{i \in [n-1] \:\: : \:\: \alpha_{i} > \alpha_{i+1}\}. $$
Subsets of $[n-1]$ are in a natural bijective correspondence
with compositions of $n$.
Hence we define the
{\em descent composition} of the permutation $\alpha$
to be the composition
$$    \Des(\alpha)
   =
     (s_{1}, s_{2}-s_{1}, s_{3}-s_{2}, \ldots,
       s_{k-1}-s_{k-2}, n-s_{k-1})   ,  $$
where the descent set of $\alpha$ is the set
$\{s_{1} < s_{2} < \cdots < s_{k-1}\}$.
We define a {\em pointed integer composition}
$\cv = (c_{1}, \ldots, c_{k})$ to be
a list of positive integers $c_{1}, \ldots, c_{k-1}$ and
a non-negative integer $c_{k}$ with $c_1 + \cdots + c_k = n$.
Note that the only part allowed to be $0$ is the last part.
When the last part is positive we refer to $\cv$ as
a {\em composition}.
Let $\beta(\cv)$ denote the number of permutations
$\alpha$ in $\SSSS_{n}$
with descent composition $\cv$ for $c_{k} > 0$.
If $c_{k} = 0$,
let $\beta(\cv) = 0$ for $k \geq 2$ and
$\beta(\cv) = 1$ for $k = 1$.

Define the {\rm (right) weak Bruhat order}
on the symmetric group~$\SSSS_{n}$ by
the cover relation
$$    \alpha \covered \alpha \circ (i,i+1)  $$
if $\alpha_{i} < \alpha_{i+1}$.
Observe that the smallest element is the identity
element $12 \cdots n$ and the largest is
$n \cdots 21$.

On the set of compositions of $n$ we define an order relation by
letting the cover relation be adding adjacent entries,
that is,
$$    (c_{1}, \ldots, c_{i},  c_{i+1}, \ldots, c_{k})
          \covered
      (c_{1}, \ldots, c_{i} + c_{i+1}, \ldots, c_{k}) .  $$
Observe that this poset is isomorphic to the Boolean algebra $B_{n}$
on $n$ elements and the maximal and minimal elements
are the two compositions
$(n)$ and $(1, \ldots, 1 , 0)$.

An {\em integer partition} $\lambda$ of a non-negative integer $n$ is
a multiset of positive integers whose sum is $n$.
We will indicate
multiplicities with a superscript.
Thus
$\{5,3,3,2,1,1,1\} = \{5,3^{2},2,1^{3}\}$
is a partition of $16$.
A {\em pointed integer partition} $(\lambda,m)$ of $n$
is a pair where $m$ is a non-negative integer and
$\lambda$ is a partition of $n-m$.
We write this as
$\{\lambda_{1}, \ldots, \lambda_{p},\underline{m}\}$
where $\lambda = \{\lambda_{1}, \ldots, \lambda_{p}\}$
is the partition and $m$ is the pointed part.
This notion of pointed integer partition
is related to pointed set partitions by
defining the type of
a pointed set partition
$\pi = \{B_{1}, B_{2}, \ldots, B_{k}, \underline{Z}\}$
to be the pointed integer partition
$$  \type(\pi)
   =
    \{|B_{1}|, |B_{2}|, \ldots, |B_{k}|, \underline{|Z|}\} . $$
Similarly, the type of a composition
$\cv = (c_{1}, \ldots, c_{k})$
is the pointed integer partition
$$  \type(\cv)
   =
    \{c_{1}, \ldots, c_{k-1}, \underline{c_{k}}\} . $$

Recall that the {\em M\"obius function} of a poset $P$
is defined by
the initial condition
$\mu(x,x) = 1$
and the recursion
$\mu(x,z) = - \sum_{x \leq y < z} \mu(x,y)$.
When the poset $P$ has a minimal element $\hz$
and a maximal element $\ho$, we call
the quantity $\mu(P) = \mu(\hz,\ho)$
the M\"obius function of the poset $P$.

For a poset $P$ define its {\em order complex}
to be the simplicial complex $\Delta(P)$
where the vertices of the complex $\Delta(P)$ are
the elements of the poset $P$ and
the faces are the chains in the poset.
In other words, the order complex of $P$ is given by
$$ \Delta(P)
      =
   \{\{x_{1}, x_{2}, \ldots, x_{k}\} \:\: : \:\:
       x_{1} < x_{2} < \cdots < x_{k}, \:
       x_{1}, \ldots, x_{k} \in P
        \}  . $$
As a consequence of Hall's Theorem
\cite[Proposition~3.8.5]{Stanley_EC_I}
we have that the reduced Euler characteristic
of the order complex~$\Delta(P)$, that is,
$\widetilde{\chi}(\Delta(P))$, is given by
the M\"obius function $\mu(P \cup \{\hz,\ho\})$,
where $P \cup \{\hz,\ho\}$ denotes the poset~$P$
with new minimal and maximal elements adjoined.

Shelling and discrete Morse theory are two powerful tools for
determining the homotopy type of simplicial complexes.
See~\cite{Forman,Forman_2,Kozlov} for more details.
We begin by a review of shellings.
A simplicial complex is {\em pure} if all its facets
(maximal faces) have the same dimension.
A pure simplicial complex is {\em shellable}
if either it is a collection of disjoint vertices
or there is an ordering on the facets $F_{1}, \ldots, F_{m}$
such that
for $2 \leq i \leq m$ the intersection
$F_{i} \cap (F_{1} \cupdots F_{i-1})$ is a pure subcomplex
of $F_{i}$ of dimension $\dim(F_{i})-1$.
A facet $F_{i}$ is called {\em spanning}
if the intersection
$F_{i} \cap (F_{1} \cupdots F_{i-1})$
is the boundary of~$F_{i}$.
\begin{theorem}
A shellable simplicial complex of dimension $d$
is homotopy equivalent to a wedge of $b$ $d$-dimensional
spheres where $b$ is the number of spanning facets
in the shelling.
\end{theorem}

Next we review discrete Morse theory.
\begin{definition}
A {\em partial matching} in a poset $P$ is a partial matching
in the underlying graph of the Hasse diagram of $P$, that is, a
subset $M \subseteq P\times P$ such that $(x,y) \in M$ implies
$x \covered y$ and each $x\in P$ belongs to at most one element
of $M$. For a pair $(x,y)$ in the matching $M$ we write
$x=d(y)$ and $y=u(x)$,
where $d$ and $u$ stand for down and up, respectively.
\end{definition}

\begin{definition}
A partial matching $M$ on $P$ is {\em acyclic} if there does not exist a
cycle
$$    z_{1} \covered u(z_{1}) \covering
      z_{2} \covered u(z_{2}) \covering
            \cdots            \covering
      z_{n} \covered u(z_{n}) \covering z_{1} , $$
of elements
in $P$ with $n\geq 2$ and all $z_{i} \in P$ distinct.
Given an acyclic matching, the unmatched
elements are called {\em critical}.
\end{definition}

We need the following version of the main theorem
of discrete Morse theory.
\begin{theorem}
Let $\Gamma$ be a simplicial complex with an acyclic matching
on its face poset, where the empty face (set) is included.
Assume that there are $b$ critical cells and that they all have
the same dimension $k$.
Then the simplicial complex $\Gamma$ is homotopy equivalent
to a wedge of $b$ spheres of dimension~$k$.
\end{theorem}

For the remainder of this section,
we restrict ourselves to considering compositions of $n$
where the last part is positive.
This restriction will also hold in
Sections~\ref{section_cycles},
\ref{section_group},
\ref{section_knapsack_cycles}
and~\ref{section_knapsack_group}.
Such a composition lies in the interval
from $(1, \ldots, 1)$ to $(n)$.
This interval is isomorphic to the Boolean algebra $B_{n-1}$
which is a complemented lattice. Hence for such
a composition $\cv$ there exists a complementary composition~$\cvc$
such that
$\cv \meet \cvc = (1, \ldots, 1)$
and
$\cv \join \cvc = (n)$.
As an example, the complement of
the composition
$(1,3,1,1,4) = (1,1+1+1,1,1,1+1+1+1)$
is obtained by exchanging
commas and plus signs, that is,
$(1+1,1,1+1+1+1,1,1,1) = (2,1,4,1,1,1)$.
Note that the complementary composition has
$n-k+1$ parts.

For a composition $\cv = (c_{1}, \ldots, c_{k})$
define the intervals $R_{1}, \ldots, R_{k}$ by
$R_{i} = [c_{1} + \cdots + c_{i-1} + 1,
          c_{1} + \cdots + c_{i}]$.
Define the subgroup~$\SSSS_{\cv}$ of the symmetric
group~$\SSSS_{n}$ by
$$   \SSSS_{\cv}
   =
     \SSSS_{R_{1}}
       \timesdots
     \SSSS_{R_{k}}   .  $$
Let $K_{1}, \ldots, K_{n-k+1}$
be the corresponding intervals
for the complementary composition
$\cvc$. Define the subgroup~$\SSSS_{\cv}^{c}$ by
$$   \SSSS_{\cv}^{c}
   =
     \SSSS_{\cvc}
   =
     \SSSS_{K_{1}}
       \timesdots
     \SSSS_{K_{n-k+1}}   .  $$

A {\em border strip} is a connected skew shape which does not
contain a $2$ by $2$ square~\cite[Section~7.17]{Stanley_EC_II}.
For each composition $\cv$
there is a unique border strip $B$
that has $k$ rows and
the $i$th row from below consists of $c_{i}$ boxes.
If we label the $n$ boxes of the border strip
from southwest to northeast, then the intervals
$R_{1}, \ldots, R_{k}$ correspond to the rows
and the intervals
$K_{1}, \ldots, K_{n-k+1}$ correspond to the columns.
Furthermore, the group
$\SSSS_{\cv}$ is the row stabilizer
and the group~$\SSSS_{\cv}^{c}$
is the column stabilizer
of the border strip $B$.

\section{Two subposets of the pointed partition lattice}
\label{section_subposets}

We now define the first poset central to this paper.
\begin{definition}
For $\cv$ a composition of $n$,
let $\Pi^{\bullet}_{\cv}$ be the subposet
of the pointed partition lattice $\Pi^{\bullet}_{n}$
described by
$$   \Pi^{\bullet}_{\cv}
   =
     \left\{\pi \in \Pi^{\bullet}_{n}
             \:\: : \:\:
       \exists \dv \geq \cv, \:\: \type(\pi) = \type(\dv) \right\} . $$
\end{definition}
In other words,
the poset
$\Pi^{\bullet}_{\cv}$ consists of all pointed set
partitions such that their type
is the type of some composition $\dv$ which is greater than
or equal to the composition $\cv$ in the composition order.

\begin{example}
{\rm
Consider the composition $\cv = (d, \ldots,d, d-1)$
of the integer $n = d \cdot k - 1$.
For a composition to be greater than or equal to $\cv$,
all of its parts must be divisible by~$d$
except the last part which is congruent to $d-1$ modulo $d$.
Hence $\Pi^{\bullet}_{\cv}$ consists of all
pointed set partitions where the block sizes
are divisible by $d$ except the zero block
whose size is congruent to $d-1$ modulo~$d$.
Hence the poset $\Pi^{\bullet}_{\cv}$ is isomorphic
to the $d$-divisible partition lattice $\Pi_{n+1}^{d}$.
}
\label{example_divisible_partition_lattice}
\end{example}

\begin{example}
{\rm
We note that $\Pi^{\bullet}_{\cv} \cup \{\hz\}$
is in general not a lattice.
Consider the composition $\cv = (1,1,2,1)$
and the four pointed set partitions
$$  \pi_{1} = 1|2|34|\underline{5}, \:\:
    \pi_{2} = 2|5|34|\underline{1}, \:\:
    \pi_{3} = 34|\underline{125}    \:\: \text{ and } \:\:
    \pi_{4} = 2|\underline{1345}  $$
in $\Pi^{\bullet}_{(1,1,2,1)}$.
In the pointed partition lattice $\Pi^{\bullet}_{5}$
we have
$\pi_{1}, \pi_{2} < 2|34|\underline{15} < \pi_{3}, \pi_{4}$.
Since the pointed set partition $2|34|\underline{15}$
does not belong to $\Pi^{\bullet}_{(1,1,2,1)}$, we conclude
that
$\Pi^{\bullet}_{(1,1,2,1)} \cup \{\hz\}$ is not a lattice.
}
\end{example}

We now turn our attention to filters in the pointed
partition lattice $\Pi^{\bullet}_{n}$
that are generated by a pointed knapsack partition.
These filters were introduced in~\cite{Ehrenborg_Readdy_I}.

Recall that we view an integer partition $\lambda$ as a multiset
of positive integers.
Let
$\lambda = \{\lambda_{1}^{e_{1}}, \ldots, \lambda_{q}^{e_{q}}\}$
be an integer partition, where we assume
that the $\lambda_{i}$'s are distinct.
If all the $(e_{1}+1) \cdots (e_{q}+1)$ integer linear combinations
$$   \left\{ \sum_{i=1}^{q} f_{i} \cdot \lambda_{i}
               \:\: : \:\:
             0 \leq f_{i} \leq e_{i} \right\} $$
are distinct, we call $\lambda$ a {\em knapsack partition}.
A pointed integer partition $\{\lambda, \underline{m}\}$ is called
a {\em pointed knapsack partition}
if the partition $\lambda$ is a knapsack partition.

This definition was introduced by
Ehrenborg--Readdy~\cite{Ehrenborg_Readdy_I}.
Their motivation was to compute the M\"obius function
of filters generated by knapsack partitions
in the pointed partition lattice;
see Corollary~\ref{corollary_Ehrenborg_Readdy}.
However, earlier Kozlov~\cite{Kozlov_paper} introduced the same
notion under the name no equal-subsets sums (NES).
His motivation was the same, except he studied the
topology of filters in the partition lattice.

\begin{definition}
For a pointed knapsack partition $\{\lambda, \underline{m}\}
=\{\lambda_{1}, \lambda_{2}, \ldots, \lambda_{k}, \underline{m}\}$
of $n$
define the subposet~$\Pi^{\bullet}_{\{\lambda, \underline{m}\}}$
to be the filter
of $\Pi^{\bullet}_{n}$
generated by all pointed set partitions of type
$\{\lambda, \underline{m}\}$.
\end{definition}

\begin{example}
{\rm Observe that
$\lambda = \{d,d, \ldots, d\}$ is a knapsack
partition. Hence all the block sizes
in the filter
$\Pi^{\bullet}_{\{\lambda, \underline{m}\}}$
are divisible by $d$, except the pointed block.
Hence for the pointed knapsack partition
$\{d,d, \ldots, d, \underline{d-1}\}$
we obtain the $d$-divisible partition lattice again,
as in
Example~\ref{example_divisible_partition_lattice}.
}
\end{example}

Note that $\Pi^{\bullet}_{\{\lambda, \underline{m}\}} \cup \{\hz\}$
is indeed a lattice since it inherits the join operation from
$\Pi^{\bullet}_{n}$. The fact the meet exists is due to
Proposition~3.3.1 in~\cite{Stanley_EC_I}.

\begin{example}
{\rm
To see the difference between the two
subposets defined in this section,
consider the pointed knapsack partition
$\{3,1,1,\underline{0}\}$ and
the composition $(3,1,1,0)$.
Observe that the pointed set partition
$\pi = 1/2/\underline{345}$ belongs to
the filter $\Pi^{\bullet}_{\{3,1,1, \underline{0}\}}$.
However, $\pi$ does not belong to
the subposet $\Pi^{\bullet}_{(3,1,1, 0)}$.
To observe this fact, note that the cardinality of
the pointed block is $3$
and there is no composition
$\dv$ greater than or equal to $(3,1,1,0)$
whose last part is $3$.
}
\end{example}

\section{The simplicial complex of ordered set partitions}
\label{section_Delta}

An {\em ordered set partition} $\tau$ of a set $S$ is
a list of blocks $(C_{1}, C_{2}, \ldots, C_{m})$
where the blocks are subsets of the set $S$
satisfying:
\begin{itemize}
\item[(i)]
All blocks except possibly the last block are non-empty,
that is, $C_{i} \neq \emptyset$ for $1 \leq i \leq m-1$.
\item[(ii)]
The blocks are pairwise disjoint,
that is, $C_{i} \cap C_{j} = \emptyset$
for $1 \leq i < j \leq m$.
\item[(iii)]
The union of the blocks is $S$,
that is, $C_{1} \cupdots C_{m} = S$.
\end{itemize}
To distinguish from pointed partitions we write
$36$-$127$-$8$-$45$ for $(\{3,6\},\{1,2,7\},\{8\},\{4,5\})$.
The {\em type} of an ordered set partition, $\type(\tau)$,
is the composition
$(|C_{1}|, |C_{2}|, \ldots, |C_{m}|)$.

Let $\Delta_{n}$ denote the collection of all ordered set partitions
of the set $[n]$.
We view $\Delta_{n}$ as a simplicial complex.
The ordered set partition
$\tau = (C_{1}, C_{2}, \ldots, C_{m})$
forms an $(m-2)$-dimensional face.
It has $m-1$ facets, which are
$(C_{1}, \ldots, C_{i} \cup C_{i+1}, \ldots, C_{m})$
for $1 \leq i \leq m-1$.
The empty face corresponds to
the ordered partition $([n])$.
The complex $\Delta_{n}$ has $2^{n} - 1$ vertices
that are of the form $(C_{1}, C_{2})$ where $C_{1} \neq \emptyset$.
Moreover there are $n!$ facets corresponding
to permutations in the symmetric group~$\SSSS_{n}$,
that is, for a permutation
$\alpha=\alpha_1 \cdots \alpha_n$, the associated facet is
$(\{\alpha_{1}\}, \{\alpha_{2}\}, \ldots, \{\alpha_{n}\}, \emptyset)$.

The permutahedron is the $(n-1)$-dimensional polytope
obtained by taking the convex hull of the~$n!$ points
$(\alpha_{1}, \ldots, \alpha_{n})$ where
$\alpha=\alpha_1 \cdots \alpha_n$ ranges over all permutations
in the symmetric
group~$\SSSS_{n}$.
Let~$P_{n}$ denote the boundary complex
of the dual of the $(n-1)$-dimensional permutahedron.
Since the permutahedron is a simple polytope
the complex $P_{n}$ is a simplicial complex
homeomorphic to an $(n-2)$-dimensional sphere.
Another view is that $P_{n}$ is the barycentric
subdivision of the boundary of the $n$-dimensional simplex.
Note that
the link of the vertex $([n],\emptyset)$
in the complex $\Delta_{n}$ is
the complex~$P_{n}$.
In fact, the complex $\Delta_{n}$ is the cone of $P_{n}$.

For a permutation $\alpha = \alpha_{1} \cdots \alpha_{n}$
in the symmetric group~$\SSSS_{n}$
and a composition $\cv = (c_{1}, \ldots, c_{k})$
of~$n$, define the ordered partition
\begin{eqnarray*}
     \sigma(\alpha,\cv)
  & = &
     \left(
       \{\alpha_{j} \:\: : \:\: j \in R_{i}\}
     \right)_{1 \leq i \leq k} \\
  & = &
     (\{\alpha_{1}, \ldots, \alpha_{c_{1}}\},
      \{\alpha_{c_{1}+1}, \ldots, \alpha_{c_{1}+c_{2}}\},
         \ldots,
      \{\alpha_{c_{1}+\cdots+c_{k-1}+1}, \ldots, \alpha_{n}\}) .
\end{eqnarray*}
We write $\sigma(\alpha)$ when it is clear from
the context what the composition $\cv$ is.

For a composition $\cv$ define the subcomplex $\Delta_{\cv}$
to be
$$   \Delta_{\cv}
  =
   \{\tau \in \Delta_{n} \:\: : \:\: \cv \leq \type(\tau)\} .  $$
This complex has all of its facets of type $\cv$.
Especially, each facet has the form $\sigma(\alpha,\cv)$
for some permutation $\alpha$.
As an example, note that
$\Delta_{(1,1, \ldots, 1)}$ is the complex $P_{n}$.
Two more examples are shown in Figures~\ref{figure_121}
and~\ref{figure_211}.

\begin{lemma}
If the pointed composition $\cv = (c_{1}, \ldots, c_{k})$ ends with $0$,
then the simplicial complex $\Delta_{\cv}$ is
a cone over the complex $\Delta_{(c_{1}, \ldots, c_{k-1})}$
with apex $([n],\emptyset)$
and hence contractible.
\label{lemma_last_part_zero}
\end{lemma}

However for a facet $F$ in $\Delta_{\cv}$ there are
$\cv\,! = c_{1}! \cdots c_{k}!$
permutations that map to it by the function~$\sigma$. Hence let
$\sigma^{-1}(F)$ denote the smallest permutation $\alpha$
with respect to the weak Bruhat order
that gets mapped to the facet $F$.
This permutation satisfies the inequalities
$$  \alpha_{c_{1}+\cdots+c_{i}+1}
    < \cdots <
    \alpha_{c_{1}+\cdots+c_{i+1}}  $$
for $0 \leq i \leq k-1$.
Furthermore, the descent composition of
the permutation $\sigma^{-1}(F)$
is greater than or equal to the composition $\cv$,
that is,
$\Des(\sigma^{-1}(F)) \geq \cv$.

\begin{figure}[t]
\setlength{\unitlength}{0.7mm}
\begin{center}
\begin{picture}(120,120)(0,0)

\put(0,0){\circle*{3}}
\put(120,0){\circle*{3}}
\put(0,120){\circle*{3}}
\put(120,120){\circle*{3}}
\put(40,40){\circle*{3}}
\put(80,40){\circle*{3}}
\put(40,80){\circle*{3}}
\put(80,80){\circle*{3}}

\thicklines

\put(0,0){\line(1,0){120}}
\put(0,0){\line(0,1){120}}
\put(120,120){\line(-1,0){120}}
\put(120,120){\line(0,-1){120}}

\put(40,40){\line(1,0){40}}
\put(40,40){\line(0,1){40}}
\put(80,80){\line(-1,0){40}}
\put(80,80){\line(0,-1){40}}

\put(0,0){\line(1,1){40}}
\put(120,0){\line(-1,1){40}}
\put(120,120){\line(-1,-1){40}}
\put(0,120){\line(1,-1){40}}

\put(-5,-6){\small 4-123}
\put(115,-6){\small 234-1}
\put(115,123){\small 3-124}
\put(-5,123){\small 134-2}

\put(26,40){\small 124-3}
\put(82,40){\small 2-134}
\put(82,78){\small 123-4}
\put(26,78){\small 1-234}

\put(55,-5){\small 4-23-1}
\put(55,35){\small 2-14-3}
\put(55,82){\small 1-23-4}
\put(55,122){\small 3-14-2}

\put(13,17){\rotatebox{45}{\small 4-12-3}}
\put(96,26){\rotatebox{-45}{\small 2-34-1}}
\put(92,96){\rotatebox{45}{\small 3-12-4}}
\put(16,106){\rotatebox{-45}{\small 1-34-2}}

\put(-4,53){\rotatebox{90}{\small 4-13-2}}
\put(36,53){\rotatebox{90}{\small 1-24-3}}
\put(82,66){\rotatebox{-90}{\small 2-13-4}}
\put(122,66){\rotatebox{-90}{\small 3-24-1}}

\end{picture}
\end{center}
\caption{The simplicial complex $\Delta_{(1,2,1)}$
of ordered partitions. Note that the ordered
partition $1234$ corresponds to the empty face.}
\label{figure_121}
\end{figure}
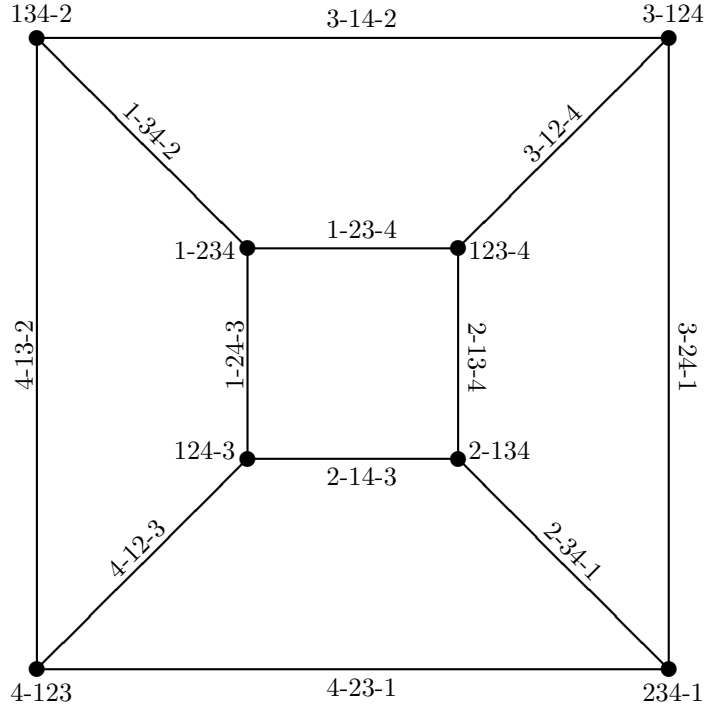

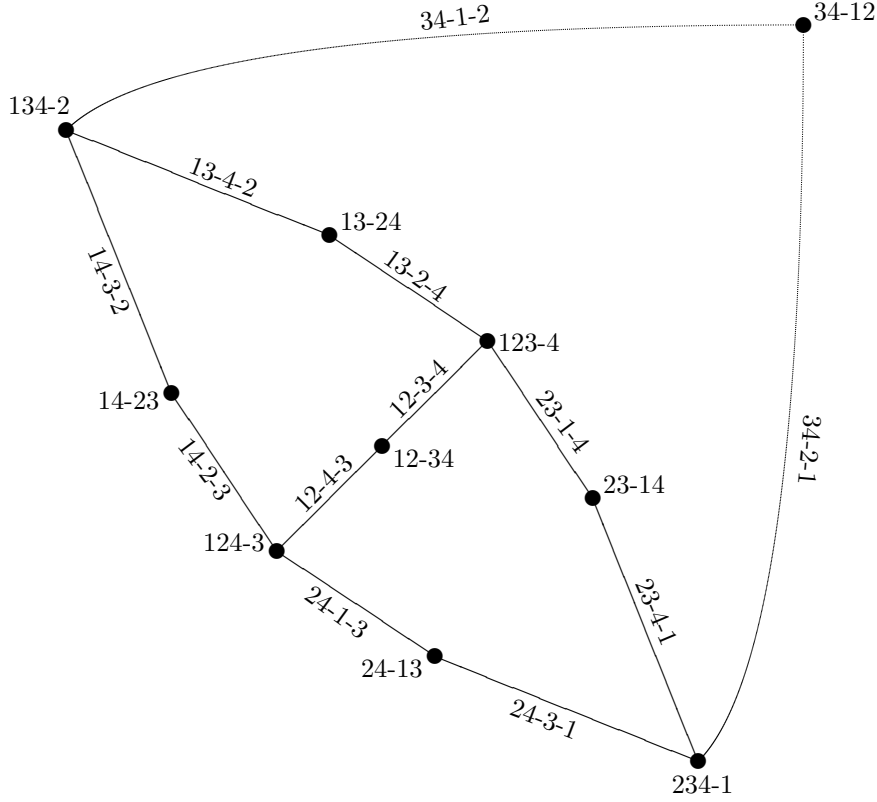
\begin{figure}[t]
\setlength{\unitlength}{0.7mm}
\begin{center}
\begin{picture}(140,140)(0,0)

\put(120,0){\circle*{3}}
\put(0,120){\circle*{3}}
\put(40,40){\circle*{3}}
\put(80,80){\circle*{3}}
\put(70,20){\circle*{3}}
\put(100,50){\circle*{3}}
\put(50,100){\circle*{3}}
\put(20,70){\circle*{3}}
\put(60,60){\circle*{3}}
\put(140,140){\circle*{3}}

\put(115,-6){\small 234-1}
\put(-11,123){\small 134-2}
\put(26,40){\small 124-3}
\put(82,78){\small 123-4}

\put(56,16){\small 24-13}
\put(6,67){\small 14-23}
\put(52,101){\small 13-24}
\put(102,51){\small 23-14}
\put(62,56){\small 12-34}
\put(142,141){\small 34-12}

\put(40,40){\line(1,1){40}}
\put(40,40){\line(3,-2){30}}
\put(40,40){\line(-2,3){20}}
\put(80,80){\line(-3,2){30}}
\put(80,80){\line(2,-3){20}}
\put(120,0){\line(-5,2){50}}
\put(120,0){\line(-2,5){20}}
\put(0,120){\line(5,-2){50}}
\put(0,120){\line(2,-5){20}}

\qbezier(120,0)(140,20)(140,140)
\qbezier(0,120)(20,140)(140,140)

\put(84,9){\rotatebox{-21}{\small 24-3-1}}
\put(45,31){\rotatebox{-36}{\small 24-1-3}}
\put(108,34){\rotatebox{-67}{\small 23-4-1}}
\put(89,69){\rotatebox{-53}{\small 23-1-4}}

\put(21,60){\rotatebox{-53}{\small 14-2-3}}
\put(4,97){\rotatebox{-67}{\small 14-3-2}}
\put(60,95){\rotatebox{-36}{\small 13-2-4}}
\put(23,112){\rotatebox{-21}{\small 13-4-2}}

\put(43,47){\rotatebox{45}{\small 12-4-3}}
\put(61,65){\rotatebox{45}{\small 12-3-4}}
\put(67,139){\rotatebox{7}{\small 34-1-2}}
\put(139,66){\rotatebox{-95}{\small 34-2-1}}

\end{picture}
\end{center}
\caption{The simplicial complex $\Delta_{(2,1,1)}$
of ordered partitions.}
\label{figure_211}
\end{figure}

\begin{theorem}
Let $\cv$ be a composition not ending with a zero.
Then the simplicial complex $\Delta_{\cv}$ is shellable.
The spanning facets are of the form
$\sigma(\alpha)$ where $\alpha$ ranges over all
permutations in the symmetric group~$\SSSS_{n}$
with descent composition $\cv$, that is,
$\Des(\alpha) = \cv$.
Hence the complex $\Delta_{\cv}$ is homotopy equivalent to a wedge
of $\beta(\cv)$ spheres of dimension $k-2$.
\label{theorem_shelling}
\end{theorem}
\begin{proof}
Let $S$ be the subset of $[n-1]$ associated with the
composition $\cv$, that is,
$S = \{c_{1}, c_{1} + c_{2}, \ldots, c_{1} + \cdots + c_{k-1}\}$.
A facet $(C_{1}, \ldots, C_{k})$ of the complex
$\Delta_{\cv}$ corresponds to the maximal chain
$$    \emptyset \subseteq
      C_{1} \subseteq
      C_{1} \cup C_{2} \subseteq \cdots \subseteq
      C_{1} \cupdots C_{k-1} \subseteq
      [n]  $$
of the $S$-rank-selected Boolean algebra
$B_{n}(S) = \{x \in B_{n} : \rho(x) \in S\} \cup \{\hz,\ho\}$.
Hence $\Delta_{\cv}$ is the order complex
$\Delta(B_{n}(S) - \{\hz,\ho\})$.
The Boolean algebra has an $EL$-labeling,
which implies that the rank-selected
poset $B_{n}(S)$ has a shellable order complex;
see~\cite[Theorem~4.1]{Bjorner_I}.
In fact, the rank-selected poset $B_{n}(S)$ is $CL$-shellable;
see~\cite[Theorem~8.1]{Bjorner_Wachs}
or~\cite[Theorem~3.4.1]{Wachs_III}.
Furthermore, it follows from Bj\"orner's
construction that the spanning facets are
exactly of the above form.
\end{proof}

Finally, we note the following consequence.
\begin{corollary}
If the pointed composition $\cv = (c_{1}, \ldots, c_{k})$ ends with $0$,
then the simplicial complex~$\Delta_{\cv}$ is shellable.
\label{corollary_last_part_zero}
\end{corollary}
\begin{proof}
The statement follows from the fact that the complex is the cone over
a shellable complex.
\end{proof}

\section{The homotopy type of the poset $\Pi^{\bullet}_{\cv}$}
\label{section_Quillen}

We now will use Quillen's Fiber Lemma to show that
the chain complex $\Delta(\Pi^{\bullet}_{\cv}-\{\ho\})$
is homotopy equivalent to
the simplicial complex $\Delta_{\cv}$.
Recall that a {\em simplicial map} $f$
from a simplicial complex $\Gamma$ to a poset $P$
sends vertices of $\Gamma$ to elements of $P$
and faces of the simplicial complex to chains of $P$.
We have the following result due to
Quillen~\cite{Quillen}.
See also~\cite[Theorem~10.5]{Bjorner_handbook}
and~\cite[Theorem~5.2.1]{Wachs_III}.

\begin{theorem}[Quillen's Fiber Lemma]
Let $f$ be a simplicial map from the
simplicial complex $\Gamma$ to the poset $P$
such that for all $x$ in $P$, the complex
$\Delta(f^{-1}(P_{\geq x}))$,
that is, the subcomplex of $\Gamma$ induced
by $f^{-1}(P_{\geq x})$, is contractible.
Then the order complex $\Delta(P)$ and
the simplicial complex $\Gamma$ are homotopy equivalent.
\end{theorem}

Recall that the {\em barycentric subdivision}
of a simplicial complex $\Gamma$
is the simplicial complex $\sd(\Gamma)$
whose vertices are the non-empty faces of $\Gamma$
and whose faces are subsets of chains of
faces in $\Gamma$ ordered by inclusion.
It is well-known that
$\Gamma$ and $\sd(\Gamma)$ are homeomorphic
since they have the same geometric realization
and hence are homotopy equivalent.

Consider the map
$\phi : \Delta_{n} \longrightarrow \Pi^{\bullet}_{n}$
that sends
an ordered set partition $(C_{1}, C_{2}, \ldots, C_{k})$
to the pointed partition
$\{C_{1}, C_{2}, \ldots, C_{k-1}, \underline{C_{k}}\}$.
We call this map the {\em forgetful} map
since it forgets the order between the blocks
except it keeps the last part as the pointed block.
Observe that the inverse image of
the pointed partition
$\{C_{1}, C_{2}, \ldots, C_{k-1}, \underline{C_{k}}\}$
consists of $(k-1)!$ ordered set partitions.

\begin{lemma}
Let $\pi$ be the pointed partition
$\{B_{1}, \ldots, B_{m-1}, \underline{B_{m}}\}$
in $\Pi^{\bullet}_{\cv}$ where $m \geq 2$.
Let~$\Omega$ be the subcomplex of the complex $\Delta_{\cv}$
whose faces are given by the inverse image
$$
  \phi^{-1}\left(\left(\Pi^{\bullet}_{\cv}-\{\ho\}\right)_{\geq \pi}\right) .
$$
Then the complex $\Omega$ is a cone over
the apex
$([n] - B_{m}, B_{m})$
and hence is contractible.
\label{lemma_Omega}
\end{lemma}
\begin{proof}
Let $\varsigma = (C_{1}, C_{2}, \ldots, C_{r})$
be an ordered partition in the complex $\Omega$.
Observe that the pointed block $B_{m}$ is contained in
the last block $C_{r}$.
Let $\Omega^{\prime}$ be the subcomplex of
$\Omega$ consisting of all faces
$(C_{1}, C_{2}, \ldots, C_{r})$
such that the last block~$C_{r}$ strictly contains
the pointed block~$B_{m}$.
Then the complex~$\Omega$ is a cone
over the complex~$\Omega^{\prime}$
with the apex $([n] - B_{m}, B_{m})$.
\end{proof}

\begin{theorem}
The order complex $\Delta(\Pi^{\bullet}_{\cv}-\{\ho\})$
is homotopy equivalent to
the barycentric subdivision~$\sd(\Delta_{\cv})$
and hence $\Delta_{\cv}$.
\label{theorem_pointed_partitions}
\end{theorem}
\begin{proof}
Note that
$$\phi^{-1}\left(\left(\Pi^{\bullet}_{\cv}-\{\ho\}\right)_{\geq \pi}\right)
    =
 \sd(\Omega)
    \simeq
 \Omega, $$
which is contractible by Lemma~\ref{lemma_Omega}.
Hence Quillen's Fiber Lemma applies
and we conclude that
$\Delta\left(\Pi^{\bullet}_{\cv}-\{\ho\}\right)$
is homotopy equivalent to
$\sd(\Delta_{\cv})$.
\end{proof}

By considering the reduced Euler characteristic of the complex
$\Delta\left(\Pi^{\bullet}_{\cv}-\{\ho\}\right)$,
we have the following corollary.
\begin{corollary}
The M\"obius function of the poset
$\Pi^{\bullet}_{\cv} \cup \{\hz\}$ is given by $(-1)^{k} \cdot \beta(\cv)$.
\end{corollary}
We note that a combinatorial proof can be given for this corollary,
which avoids the use of Quillen's Fiber Lemma.

\section{Cycles in the complex $\Delta_{\cv}$}
\label{section_cycles}

In this section and the next we assume that the last part of
the composition $\cv$ is non-zero, since in the case
$c_{k} = 0$ the top homology group is the trivial group;
see Lemma~\ref{lemma_last_part_zero}.

For $\alpha$ a permutation in the symmetric group
$\SSSS_{n}$, define the subcomplex
$\Sigma_{\alpha}$ of the complex $\Delta_{\cv}$
to be the simplicial complex
whose facets are given by
$$ \left\{\sigma(\alpha \circ \gamma)
               \:\: : \:\:
      \gamma \in \SSSS_{\cv}^{c} \right\} , $$
where $\sigma$ is defined in Section~\ref{section_Delta}.

\begin{lemma}
The subcomplex $\Sigma_{\alpha}$ is
isomorphic to the join
of the duals of the permutahedra
$$P_{|K_{1}|} *  \cdots * P_{|K_{n-k+1}|}$$
and hence it is a sphere of dimension $k-2$.
\label{lemma_sphere_I}
\end{lemma}
\begin{proof}
Represent the complex $P_{|K_{i}|}$
by ordered partitions on the set $K_{i}$.
A face of the join
$P_{|K_{1}|} *  \cdots * P_{|K_{n-k+1}|}$
is then an $(n-k+1)$-tuple
$(\tau_{1}, \ldots, \tau_{n-k+1})$
where $\tau_{i} = (D_{i,1}, \ldots, D_{i,j_{i}}) \in P_{|K_{i}|}$.
We obtain a face of
$\Sigma_{\alpha}$ by gluing
these ordered partitions together, that is,
$$ (D_{1,1}, \ldots, D_{1,j_{1}} \cup
    D_{2,1}, \ldots, D_{2,j_{2}} \cup
    D_{3,1}, \ldots, D_{n-k+1,j_{n-k+1}}) . $$
It is straightforward to see this is a bijective
correspondence proving the first claim.
Since the join of an $m$-sphere and an $n$-sphere
is an $(m+n+1)$-sphere, we obtain a
sphere of dimension
$(|K_{1}|-2) + \cdots + (|K_{n-k+1}|-2) + n-k = k-2$.
\end{proof}

Observe that the facets of $\Delta_{\cv}$ are in bijection
with permutations $\alpha$ such that $\Des(\alpha) \geq \cv$
in the composition order.

\begin{lemma}
Let $\alpha$ be a permutation in the symmetric group
$\SSSS_{n}$ with descent composition $\cv$ and let
$\gamma$ belong to
the column stabilizer~$\SSSS_{\cv}^{c}$.
Then we have the inequality
$\alpha \circ \gamma \leq \alpha$
in the weak Bruhat order.
\label{lemma_Bruhat}
\end{lemma}
\begin{proof}
Write $\gamma$ as $(\gamma_{1}, \ldots, \gamma_{n-k+1})$
where $\gamma_{i}$ belongs to $\SSSS_{K_{i}}$.
Then $\gamma_{i}$ acts on the interval $K_{i} = [u,v]$.
Note that
$v (v-1) \cdots u$ is the largest permutation in the Bruhat
order on $\SSSS_{K_{i}}$.
Hence $\alpha \circ \gamma$ restricted to $K_{i}$
is a smaller element in the Bruhat order.
The inequality follows by
concatenating
these partial permutations
into the permutation $\alpha \circ \gamma$.
\end{proof}

We note that this lemma has a dual version.
\begin{lemma}
Let $\alpha$ be a permutation in the symmetric group
$\SSSS_{n}$ with descent composition $\cv$ and let
$\gamma$ belong to
the row stabilizer~$\SSSS_{\cv}$.
Then we have the inequality
$\alpha \circ \gamma \geq \alpha$
in the weak Bruhat order.
\end{lemma}

Directly we have the next lemma.
\begin{lemma}
Let $F$ be a facet of $\Sigma_{\alpha}$.
Then the inequality
$\sigma^{-1}(F) \leq \alpha$ holds in the weak Bruhat order.
\label{lemma_inequality}
\end{lemma}
\begin{proof}
There is an element $\gamma$
in
the column stabilizer~$\SSSS_{\cv}^{c}$
such that $F = \sigma(\alpha \circ \gamma)$.
Since $\sigma^{-1}(F)$ is the smallest permutation
in the Bruhat order that maps to $F$, we have that
$\sigma^{-1}(F) \leq \alpha \circ \gamma \leq \alpha$.
\end{proof}

Recall that the boundary map of the face
$\sigma = (C_{1}, \ldots, C_{r})$
in the chain complex of $\Delta_{\cv}$ is defined
by
$$     \partial((C_{1}, \ldots, C_{r}))
     =
       \sum_{i=1}^{r-1}
                 (-1)^{i-1}
              \cdot
                 (C_{1}, \ldots, C_{i} \cup C_{i+1}, \ldots, C_{r}) . $$

The next lemma follows from
Lemma~\ref{lemma_sphere_I}, apart from the signs.
\begin{lemma}
For $\alpha \in \SSSS_{n}$, the element
$$ g_\alpha = \sum_{\gamma \in \SSSS_{\cv}^{c}}
            (-1)^{\gamma}
               \cdot
            \sigma(\alpha \circ \gamma)     $$
in the chain group $C_{k-2}(\Delta_{\cv})$
belongs to the kernel of the boundary map
and hence to the homology group $\HH_{k-2}(\Delta_{\cv})$.
\label{lemma_boundary_map}
\end{lemma}
\begin{proof}
Apply the boundary map to the above element $g_{\alpha}$
in the chain
group and exchange the order of the two sums.
The inner sum is then
$$   \sum_{\gamma \in \SSSS_{\cv}^{c}}
            (-1)^{\gamma}
               \cdot
      (\ldots,
        \{ \ldots, \alpha_{\gamma_{c_{1}+\cdots+c_{i}}} ,
           \alpha_{\gamma_{c_{1}+\cdots+c_{i}+1}}, \ldots \},
       \ldots)   .   $$
Observe that the term corresponding to $\gamma$
cancels with the term corresponding to
$\gamma$ composed with
the transposition $(c_{1}+\cdots+c_{i}, c_{1}+\cdots+c_{i}+1)$
which belongs to the column stabilizer~$\SSSS_{\cv}^{c}$
and thus the sum vanishes.
\end{proof}

\begin{theorem}
The cycles $g_{\alpha}$,
where $\alpha$ ranges over all permutations
with descent composition $\cv$, form a basis
for the homology group $\HH_{k-2}(\Delta_{\cv})$.
\label{theorem_homology}
\end{theorem}
\begin{proof}
The complex $\Delta_{\cv} -
   \{\sigma(\alpha) : \Des(\alpha) = \cv\}$
is contractible by the shelling in
the proof of Theorem~\ref{theorem_shelling}.
Contract this complex to a point and then attach the cells
$\sigma(\alpha)$ to this point to obtain a wedge of spheres,
denoted by $X$.
Call this contraction map $f$, that is,
we have the continuous function
$f: \Delta_{\cv} \longrightarrow X$
and hence
the homomorphism
$f_{*}: \HH_{k-2}(\Delta_{\cv}) \longrightarrow \HH_{k-2}(X)$.
Since $f$ is a part of a homotopy equivalence, we know
that $f_{*}$ is an isomorphism.
Lastly, let $h_{\alpha}$ denote the cycle corresponding
to the sphere~$\sigma(\alpha)$ in the homology group
$\HH_{k-2}(X)$.

It is clear that the cycles $h_{\alpha}$,
where $\alpha$ ranges over all permutations with descent composition
$\cv$,
form a basis for the homology group
$\HH_{k-2}(X)$.
We now motivate that the images
$f_{*}(g_{\alpha})$ also form a basis for the homology group
$\HH_{k-2}(X)$.

Take two permutations $\alpha$ and $\alpha^{\prime}$
with descent composition $\cv$.
By Lemma~\ref{lemma_inequality}
we have that the coefficient of $h_{\alpha^{\prime}}$
in $f_{*}(g_{\alpha})$ is zero
unless $\alpha^{\prime} \leq \alpha$ in the weak Bruhat order.
Furthermore, the coefficient of $h_{\alpha}$
in $f_{*}(g_{\alpha})$ is $1$.
Hence the relationship between
the basis $\left\{h_{\alpha^{\prime}}\right\}_{\alpha^{\prime}}$
and the set $\left\{f_{*}(g_{\alpha})\right\}_{\alpha}$
is triangular and hence invertible.
Thus the set
$\left\{f_{*}(g_{\alpha})\right\}_{\alpha}$ forms a basis for
the homology group $\HH_{k-2}(X)$.
Finally, since~$f_{*}$ is an isomorphism,
the cycles $g_{\alpha}$ form a basis for the homology group
$\HH_{k-2}(\Delta_{\cv})$.
\end{proof}

\section{Representation of the symmetric group}
\label{section_group}

The symmetric group~$\SSSS_{n}$ acts naturally
on the poset $\Pi^{\bullet}_{\cv}$ by relabeling
the elements. Hence it also acts on the order complex
$\Delta(\Pi^{\bullet}_{\cv} - \{\ho\})$. Lastly,
the symmetric group acts on the top homology
group $\HH_{k-2}(\Delta(\Pi^{\bullet}_{\cv} - \{\ho\}))$.
We show in this section that this action is a Specht module
of the border strip $B$ corresponding to the composition $\cv$.
For an overview on the representation theory of the symmetric group,
we refer the reader to Sagan's book~\cite{Sagan}
and Wachs' article~\cite{Wachs_III}.

Let $G$ be a group which acts on
the simplicial complex $\Gamma$
and on the poset $P$.
We call the simplicial map $f$
a {\em $G$-simplicial map} if the map $f$ commutes
with this action.
The equivariant version of Quillen's Fiber Lemma
is as follows~\cite{Quillen}.
See also~\cite[Section~5.2]{Wachs_III}.
\begin{theorem}[Equivariant homology version of Quillen's Fiber Lemma]
Let $f$ be a $G$-simplicial map from the
simplicial complex $\Gamma$ to the poset $P$
such that for all $x$ in $P$, the subcomplex
$\Delta(f^{-1}(P_{\geq x}))$ is acyclic.
Then the two homology groups
$\HH_{r}(\Delta(P))$ and $\HH_{r}(\Gamma)$ are
isomorphic as $G$-modules.
\label{theorem_equivariant_Quillen}
\end{theorem}

The forgetful map $\phi$ from
$\sd(\Delta_{\cv})$ to
the order complex of the poset $\Pi^{\bullet}_{\cv} - \{\ho\}$
commutes with the action of the symmetric group~$\SSSS_{n}$.
Hence we conclude the next result.

\begin{proposition}
The two homology groups
$\HH_{k-2}(\sd(\Delta_{\cv}))$ and
$\HH_{k-2}(\Delta(\Pi^{\bullet}_{\cv} - \{\ho\}))$
are isomorphic as $\SSSS_{n}$-modules.
\label{proposition_homology_modules}
\end{proposition}

It is clear that $\HH_{k-2}(\sd(\Delta_{\cv}))$ and
$\HH_{k-2}(\Delta_{\cv})$ are isomorphic as $\SSSS_{n}$-modules.
Hence in the remainder of this section
we will study the action of the symmetric group
$\SSSS_{n}$ on
$\Delta_{\cv}$ and its action on the homology group
$\HH_{k-2}(\Delta_{\cv})$.
This is in the spirit of Wachs' work~\cite{Wachs}.

Let $B$ be the border strip that has $k$ rows where
the $i$th row consists of $c_{i}$ boxes.
Recall that a tableau is a filling of the boxes of the shape
$B$ with the integers $1$ through $n$.
A standard Young tableau is a tableau where the rows
and columns are increasing.
A tabloid is an equivalence class of tableaux under the relation
of permuting the entries in each row.
To distinguish tabloids from tableaux, only
the horizontal lines are drawn in a tabloid.
See~\cite[Section~2.1]{Sagan} for details.

Observe that there is a natural bijection
between tabloids of shape $B$ and facets of the complex~$\Delta_{\cv}$
by letting the elements in each row form a block
and letting the order of the blocks
go from lowest to highest row.
See Figure~\ref{figure_border_strip} for an example.
Let $M^{B}$ be the permutation module
corresponding to shape $B$, that is,
the linear span of all tabloids of shape $B$.
Notice that the above bijection
induces a ${\mathfrak S}_{n}$-module isomorphism between
the permutation module $M^{B}$
and the chain group $C_{k-2}(\Delta_{\cv})$.

\begin{figure}[t]
\setlength{\unitlength}{0.7mm}
\begin{center}
\begin{picture}(60,50)(0,0)

\thicklines

\put(0,0){\line(1,0){20}}
\put(0,10){\line(1,0){40}}
\put(10,20){\line(1,0){30}}
\put(30,30){\line(1,0){10}}
\put(30,40){\line(1,0){30}}
\put(30,50){\line(1,0){30}}

\put(0,0){\line(0,1){10}}
\put(10,0){\line(0,1){20}}
\put(20,0){\line(0,1){20}}
\put(30,10){\line(0,1){40}}
\put(40,10){\line(0,1){40}}
\put(50,40){\line(0,1){10}}
\put(60,40){\line(0,1){10}}

\end{picture}
\hspace{10mm}
\begin{picture}(60,50)(0,0)

\put(0,0){\line(1,0){20}}
\put(0,10){\line(1,0){40}}
\put(10,20){\line(1,0){30}}
\put(30,30){\line(1,0){10}}
\put(30,40){\line(1,0){30}}
\put(30,50){\line(1,0){30}}

\put(0,0){\makebox(10,10)[c]{9}}
\put(10,0){\makebox(10,10)[c]{5}}
\put(10,10){\makebox(10,10)[c]{1}}
\put(20,10){\makebox(10,10)[c]{6}}
\put(30,10){\makebox(10,10)[c]{4}}
\put(30,20){\makebox(10,10)[c]{10}}
\put(30,30){\makebox(10,10)[c]{2}}
\put(30,40){\makebox(10,10)[c]{7}}
\put(40,40){\makebox(10,10)[c]{3}}
\put(50,40){\makebox(10,10)[c]{8}}

\end{picture}
\hspace{10mm}
\begin{picture}(60,50)(0,0)

\put(0,0){\line(1,0){20}}
\put(0,10){\line(1,0){40}}
\put(10,20){\line(1,0){30}}
\put(30,30){\line(1,0){10}}
\put(30,40){\line(1,0){30}}
\put(30,50){\line(1,0){30}}
\put(0,0){\line(0,1){10}}
\put(10,0){\line(0,1){20}}
\put(20,0){\line(0,1){20}}
\put(30,10){\line(0,1){40}}
\put(40,10){\line(0,1){40}}
\put(50,40){\line(0,1){10}}
\put(60,40){\line(0,1){10}}

\put(0,0){\makebox(10,10)[c]{5}}
\put(10,0){\makebox(10,10)[c]{7}}
\put(10,10){\makebox(10,10)[c]{4}}
\put(20,10){\makebox(10,10)[c]{10}}
\put(30,10){\makebox(10,10)[c]{8}}
\put(30,20){\makebox(10,10)[c]{1}}
\put(30,30){\makebox(10,10)[c]{3}}
\put(30,40){\makebox(10,10)[c]{6}}
\put(40,40){\makebox(10,10)[c]{2}}
\put(50,40){\makebox(10,10)[c]{9}}

\end{picture}
\end{center}
\caption{(a) Border strip corresponding to composition $(2,3,1,1,3)$.
(b) The tabloid corresponding to the face
$5\:9-1\:4\:6-10-2-3\:7\:8$.
(c) The tableau corresponding to the permutation
5\:7\:4\:10\:8\:1\:3\:6\:2\:9.}
\label{figure_border_strip}
\end{figure}
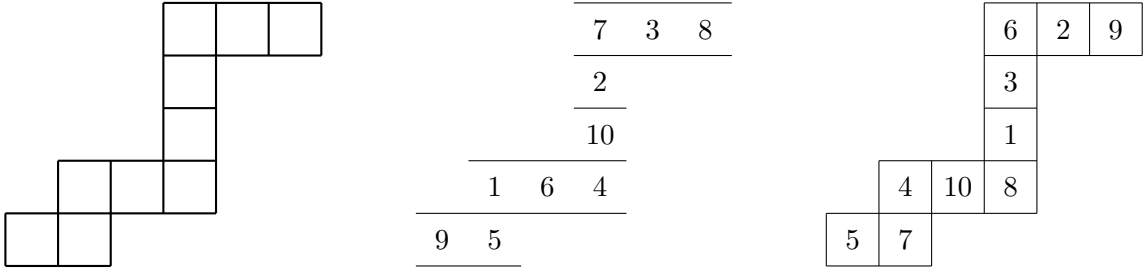

Furthermore, there is a bijection between tableaux
of shape $B$ and permutations by reading
the elements in the northeast direction from the border
strip.
Recall that the group
$\SSSS_{\cv}^{c} =
 \SSSS_{K_{1}} \timesdots \SSSS_{K_{n-k+1}}$
is the column stabilizer of the border strip $B$.
Let $t$ be a tableau and $\alpha$ its associated permutation.
Hence the polytabloid ${\bf e}_{t}$
corresponding to the tableau $t$
is the element $g_\alpha$ presented in
Lemma~\ref{lemma_boundary_map};
see~\cite[Definition~2.3.2]{Sagan}.
Since the Specht module~$S^{B}$ is the submodule of $M^{B}$ spanned
by all polytabloids,
Lemma~\ref{lemma_boundary_map}
proves that
the Specht module~$S^{B}$ is isomorphic to a submodule of
the kernel of the boundary map $\partial_{k-2}$.
Since the kernel is the top homology group
$\HH_{k-2}(\Delta_{\cv})$,
and
the Specht module~$S^{B}$
and the homology group
$\HH_{k-2}(\Delta_{\cv})$ have the same dimension $\beta(\cv)$,
we conclude that they are isomorphic.
To summarize we have:
\begin{proposition}
The top homology group $\HH_{k-2}(\Delta_{\cv})$
is isomorphic to the Specht module~$S^{B}$
as $\SSSS_{n}$-modules.
\label{proposition_Specht}
\end{proposition}

By combining Propositions~\ref{proposition_homology_modules}
and~\ref{proposition_Specht}, the main result of this section
follows.
\begin{theorem}
The top homology group $\HH_{k-2}(\Delta(\Pi^{\bullet}_{\cv} - \{\ho\}))$
is isomorphic to the Specht module~$S^{B}$
as $\SSSS_{n}$-modules.
\label{theorem_Specht}
\end{theorem}

\section{Filters generated by knapsack partitions}

We now turn our attention to filters in the pointed
partition lattice $\Pi^{\bullet}_{n}$
that are generated by a pointed knapsack partition,
that is, the poset
$\Pi^{\bullet}_{\{\lambda,\underline{m}\}}$,
which was introduced in Section~\ref{section_subposets}.
In order to study this poset we need
the corresponding collection of ordered set partitions.
\begin{definition}
For a pointed knapsack partition
$\{\lambda, \underline{m}\}
=\{\lambda_{1}, \lambda_{2}, \ldots, \lambda_{k}, \underline{m}\}$
of $n$
define the subcomplex $\Lambda_{\{\lambda,\underline{m}\}}$
of the complex of ordered set partitions $\Delta_{n}$
by
$$   \Lambda_{\{\lambda,\underline{m}\}}
    =
\{(C_{1}, \ldots, C_{r-1}, C_{r}) \in \Delta_{n}
   \:\: : \:\:
\{C_{1}, \ldots, C_{r-1}, \underline{C_{r}} \}
\in \Pi^{\bullet}_{\{\lambda,\underline{m}\}}\} . $$
\end{definition}

For a pointed knapsack partition $\{\lambda, \underline{m}\}$
of $n$ define $F$ to be the filter in the poset of
compositions of $n$ generated by compositions
$\cv$ such that $\type(\cv) = \{\lambda, \underline{m}\}$.
Now define $V(\lambda, \underline{m})$ to be
the collection of all pointed compositions
$\cv = (c_{1}, c_{2}, \ldots, c_{r})$ in the filter $F$
such that each $c_{i}, 1\leq i \leq r-1$, is
a sum of distinct parts of the partition $\lambda$ and $c_{r}=m$.
As an example, for $\lambda=\{1,1,3,7\}$
we have
$(4,8,m) \in V(\lambda, \underline{m})$ but
$(2,10,m) \not\in V(\lambda, \underline{m})$.

For a composition $\dv$ in $V(\lambda, \underline{m})$
define $\epsilon(\dv)$ to be the composition
of type $\{\lambda, \underline{m}\}$,
where each entry~$d_{i}$ of $\dv$ has been
replaced with a decreasing list of parts of $\lambda$,
that is,
$$ \epsilon(\dv) = (\lambda_{1,1}, \ldots, \lambda_{1,t_{1}},
\ldots, \lambda_{s,1}, \ldots, \lambda_{s,t_{s}}, m) , $$
where
$\lambda_{i,1} > \lambda_{i,2} > \cdots > \lambda_{i,t_{i}}$,
$\sum_{j=1}^{t_{i}} \lambda_{i,j}=d_{i}$
and
$$ \{\lambda,\underline{m}\}
     =
   \{\lambda_{1,1}, \ldots, \lambda_{1,t_{1}},
   \ldots, \lambda_{s,1}, \ldots, \lambda_{s,t_{s}}, \underline{m}\} . $$

As an example,
for the pointed knapsack partition $\lambda=\{2,1,\underline{1}\}$
we have
$\epsilon((3,1))=(2,1,1)$, $\epsilon((2,1,1))=(2,1,1)$
and $\epsilon((1,2,1))=(1,2,1)$.
Also note $\epsilon(\dv) \leq \dv$ in the partial order of compositions.

Similar to Theorem~\ref{theorem_shelling} we have the following
topological conclusion. However, this time the tool is not shelling,
but discrete Morse theory.
\begin{theorem}
There is a Morse matching on
the simplicial complex $\Lambda_{\{\lambda,\underline{m}\}}$
such that the only critical cells are of the form
$\sigma(\alpha, \epsilon(\dv))$
where
$\dv$ ranges in the set $V(\lambda, \underline{m})$
and
$\alpha$ ranges over all permutations in the symmetric group
$\SSSS_{n}$ with descent composition $\dv$.
Hence, the simplicial complex $\Lambda_{\{\lambda,\underline{m}\}}$
is homotopy equivalent to a wedge of
$\sum_{\dv \in V(\lambda, \underline{m})} \beta(\dv)$
spheres of dimension $k-1$.
\label{theorem_knapsack}
\end{theorem}

\begin{figure}[t]
\setlength{\unitlength}{0.7mm}
\begin{center}
\begin{picture}(140,140)(0,0)

\put(0,0){\circle*{3}}
\put(120,0){\circle*{3}}
\put(0,120){\circle*{3}}
\put(120,120){\circle*{3}}
\put(40,40){\circle*{3}}
\put(80,40){\circle*{3}}
\put(40,80){\circle*{3}}
\put(80,80){\circle*{3}}

\put(70,20){\circle*{3}}
\put(100,50){\circle*{3}}
\put(50,100){\circle*{3}}
\put(20,70){\circle*{3}}
\put(60,60){\circle*{3}}
\put(140,140){\circle*{3}}

\thicklines

\put(0,0){\line(1,0){120}}
\put(0,0){\line(0,1){120}}
\put(120,120){\line(-1,0){120}}
\put(120,120){\line(0,-1){120}}

\put(40,40){\line(1,0){40}}
\put(40,40){\line(0,1){40}}
\put(80,80){\line(-1,0){40}}
\put(80,80){\line(0,-1){40}}

\put(0,0){\line(1,1){40}}
\put(120,0){\line(-1,1){40}}
\put(120,120){\line(-1,-1){40}}
\put(0,120){\line(1,-1){40}}

\put(-5,-6){\small 4-123}
\put(115,-6){\small 234-1}
\put(115,123){\small 3-124}
\put(-11,123){\small 134-2}

\put(26,40){\small 124-3}
\put(82,40){\small 2-134}
\put(82,78){\small 123-4}
\put(26,78){\small 1-234}

\put(55,-5){\small 4-23-1}
\put(55,35){\small 2-14-3}
\put(55,82){\small 1-23-4}
\put(55,122){\small 3-14-2}

\put(13,17){\rotatebox{45}{\small 4-12-3}}
\put(96,26){\rotatebox{-45}{\small 2-34-1}}
\put(92,96){\rotatebox{45}{\small 3-12-4}}
\put(16,106){\rotatebox{-45}{\small 1-34-2}}

\put(-4,53){\rotatebox{90}{\small 4-13-2}}
\put(36,53){\rotatebox{90}{\small 1-24-3}}
\put(82,66){\rotatebox{-90}{\small 2-13-4}}
\put(122,66){\rotatebox{-90}{\small 3-24-1}}

\put(40,40){\line(1,1){40}}
\put(40,40){\line(3,-2){30}}
\put(40,40){\line(-2,3){20}}
\put(80,80){\line(-3,2){30}}
\put(80,80){\line(2,-3){20}}
\put(120,0){\line(-5,2){50}}
\put(120,0){\line(-2,5){20}}
\put(0,120){\line(5,-2){50}}
\put(0,120){\line(2,-5){20}}

\qbezier(120,0)(140,20)(140,140)
\qbezier(0,120)(20,140)(140,140)

\put(56,16){\small 24-13}
\put(6,67){\small 14-23}
\put(52,101){\small 13-24}
\put(102,51){\small 23-14}

\put(62,56){\small 12-34}
\put(142,141){\small 34-12}

\put(84,9){\rotatebox{-21}{\small 24-3-1}}
\put(45,31){\rotatebox{-36}{\small 24-1-3}}
\put(108,34){\rotatebox{-67}{\small 23-4-1}}
\put(89,69){\rotatebox{-53}{\small 23-1-4}}

\put(21,60){\rotatebox{-53}{\small 14-2-3}}
\put(4,97){\rotatebox{-67}{\small 14-3-2}}
\put(60,95){\rotatebox{-36}{\small 13-2-4}}
\put(23,112){\rotatebox{-21}{\small 13-4-2}}

\put(43,47){\rotatebox{45}{\small 12-4-3}}
\put(61,65){\rotatebox{45}{\small 12-3-4}}
\put(67,139){\rotatebox{7}{\small 34-1-2}}
\put(139,66){\rotatebox{-95}{\small 34-2-1}}

\end{picture}
\end{center}
\caption{The simplicial complex
$\Lambda_{\{2,1,\underline{1}\}}$,
corresponding to the knapsack partition $\{2,1,\underline{1}\}$.
Notice that this complex is the union
of two complexes $\Delta_{(1,2,1)}$ and $\Delta_{(2,1,1)}$,
appearing in Figures~\ref{figure_121} and~\ref{figure_211}.}
\label{figure_Lambda}
\end{figure}
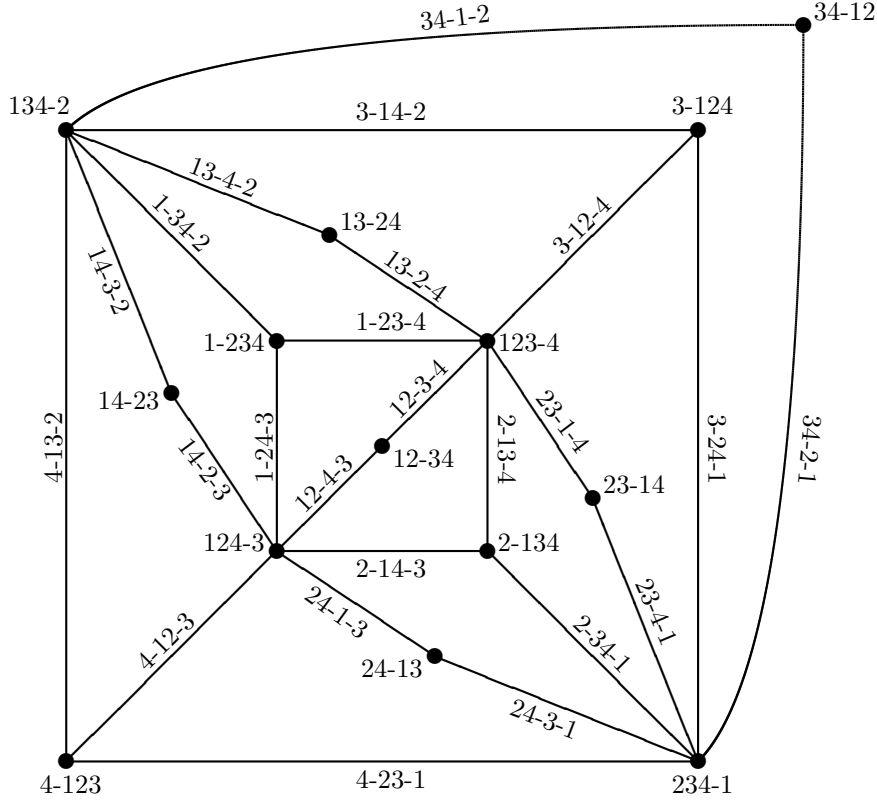

For a pointed knapsack partition
$\{ \lambda_{1}, \lambda_{2}, \ldots, \lambda_{k}, \underline{m} \}$ of $n$,
define a function $\kappa$ on the domain
$$   D
   =
     \left\{ \sum_{i \in S} \lambda_{i}
                \:\: : \:\:
             \emptyset \neq S \subseteq [k] \right\}  , $$
that is, all non-zero sums of parts of the partition,
excluding the pointed part.
Now $\kappa : D \longrightarrow \mathbb{P}$ is given by
$$    \kappa(\lambda_{i_{1}} + \lambda_{i_{2}} + \cdots + \lambda_{i_p})
    =
      \min(\lambda_{i_{1}}, \lambda_{i_{2}}, \ldots, \lambda_{i_p}) . $$
Observe that $\kappa$ is well-defined
since $\lambda$ is a knapsack partition.

For an ordered set partition $\tau = (C_{1}, C_{2}, \ldots, C_{r})$
in $\Lambda_{\{\lambda,\underline{m}\}}$, consider the following
$2r-1$ conditions:
\begin{itemize}
\item[--]
Let ${\bf A}_{i}$,
for $1 \leq i \leq r-2$,
be the condition that
$\max(C_{i}) < \min(C_{i+1})$ and $|C_{i}| \leq \kappa(|C_{i+1}|)$.

\item[--]
Let ${\bf A}_{r-1}$ be the condition that
$\max(C_{r-1}) < \min(C_{r})$.

\item[--]
Let ${\bf B}_{i}$,
for $1 \leq i \leq r-1$,
be the condition that
$\kappa(|C_{i}|) < |C_{i}|$.

\item[--]
Let ${\bf B}_{r}$ be the condition that
$|C_{r}| > m$.
\end{itemize}

Note that condition ${\bf A}_{i}$ concerns
comparing the $i$th and $(i+1)$st blocks of the set partition $\tau$.
We also notice that if
${\bf B}_{i}$ is true and $i \leq r-1$ then
the cardinality $|C_{i}|$ is
the sum of at least two parts of the partition~$\lambda$.
Similarly, if ${\bf B}_{r}$ is true then
the cardinality $|C_{r}|$ is the sum of at least
two parts of $\lambda$ and the integer $m$.

We use these conditions to construct a discrete Morse matching
for $\Lambda_{\{\lambda,\underline{m}\}}$.
We match the ordered set partition
$\tau = (C_{1}, C_{2}, \ldots, C_{r})$
in $\Lambda_{\{\lambda,\underline{m}\}}$ as follows:
\begin{itemize}
\item[--]
If conditions ${\bf A}_{j}$ and ${\bf B}_{j}$ are false
for $1 \leq j \leq i-1$ and $i \leq r-1$
but condition ${\bf B}_{i}$ is true,
then let $X$ be the $\kappa(|C_{i}|)$ smallest elements of $C_{i}$
and $Y$ be the $|C_{i}|-\kappa(|C_{i}|)$ largest elements of $C_{i}$.
Let $u(\tau)$ be given by
$$ u(\tau)
   =
   (C_{1}, C_{2}, \ldots, C_{i-1}, X, Y, C_{i+1}, \ldots, C_{r}) $$
and let the type of the edge $(\tau, u(\tau))$ be $i$.

\item[--]
If conditions ${\bf A}_{j}$ and ${\bf B}_{j}$ are false
for $1 \leq j \leq r-1$
but condition ${\bf B}_{r}$ is true,
then let $X$ be the $\kappa(|C_{r}| - m)$ smallest elements of $C_{r}$
and $Y$ be the $|C_{r}|-\kappa(|C_{r}| - m)$ largest elements of $C_{r}$.
Let $u(\tau)$ be given by
$$ u(\tau)
   =
   (C_{1}, C_{2}, \ldots, C_{r-1}, X, Y) $$
and let the type of the edge $(\tau, u(\tau))$ be $r$.

\item[--]
If conditions ${\bf A}_{j}$ and ${\bf B}_{j}$ are false
for $1 \leq j \leq i-1$
and condition ${\bf B}_{i}$ is false but condition
${\bf A}_{i}$ is true,
then let
$$ d(\tau)
   =
   (C_{1}, C_{2}, \ldots, C_{i-1}, C_{i} \cup C_{i+1},
     C_{i+2}, \ldots, C_{r}) $$
and let the type of the edge $(d(\tau), \tau)$ be $i$.
\end{itemize}

\begin{lemma}
Let $\tau$ and $\tau^{\prime}$ be two different ordered set partitions
satisfying the condition
$\tau \prec u(\tau) \succ \tau^{\prime} \prec u(\tau^{\prime})$.
Then this condition implies that
the type of $(\tau, u(\tau))$ is greater
than the type of $(\tau^{\prime}, u(\tau^{\prime}))$.
Hence the matching is acyclic.
\label{lemma_Morse}
\end{lemma}
\begin{proof}
Consider the following three ordered set partitions:
\begin{eqnarray*}
\tau
  & = &
(C_{1}, C_{2}, \ldots, C_{i-1}, C_{i} \cup C_{i+1}, C_{i+2}, \ldots, C_{r}) , \\
u(\tau)
  & = &
(C_{1}, C_{2}, \ldots, C_{r}) , \\
\tau^{\prime}
  & = &
(C_{1}, C_{2}, \ldots, C_{j-1}, C_{j} \cup C_{j+1}, C_{j+2}, \ldots, C_{r}) ,
\end{eqnarray*}
for $i \neq j$.
Note that the type of the edge $(\tau, u(\tau))$ is $i$.
If $i<j$, the ordered set partition $\tau^{\prime}$ should be matched to
$$  d(\tau^{\prime})
    =
   (C_{1}, C_{2}, \ldots, C_{i-1}, C_{i} \cup C_{i+1}, C_{i+2},
    \ldots, C_{j-1}, C_{j} \cup C_{j+1}, C_{j+2}, \ldots, C_{r}) , $$
contradicting the assumption that
$\tau^{\prime}$ was matched upwards with
$u(\tau^{\prime})$.
Hence, we conclude that $i>j$ and,
$$ u(\tau^{\prime})
     =
   (C_{1}, C_{2}, \ldots, C_{j-1}, X, Y, C_{j+2}, \ldots, C_{r}) $$
for $X \cup Y = C_{j} \cup C_{j+1}$
such that $\max(X)<\min(Y)$.
Since $\max(C_{j})>\min(C_{j+1})$,
we have $u(\tau) \neq u(\tau^{\prime})$
and the type of $(\tau^{\prime}, u(\tau^{\prime}))$ is $j$.
Therefore, we have
$\type(\tau, u(\tau)) = i > j = \type(\tau^{\prime}, u(\tau^{\prime}))$.

Assume now that this matching is not acyclic.
Then we would reach a contradiction
by following the inequalities of types around a cycle.
\end{proof}

\begin{lemma}
The critical cells of the above matching
of the face poset of $\Lambda_{\{\lambda,\underline{m}\}}$
are of the form
$\sigma(\alpha, \epsilon(\dv))$
where
$\dv$ ranges in the set $V(\lambda, \underline{m})$
and
$\alpha$ ranges over all permutations in the symmetric group
$\SSSS_{n}$ with descent composition $\dv$.
\label{lemma_cells}
\end{lemma}
\begin{proof}
The unmatched cells of the matching presented
on the face poset of $\Lambda_{\{\lambda,\underline{m}\}}$ have
the form $(C_{1}, C_{2}, \ldots, C_{r})$ such that
all the conditions
${\bf A}_{i}$ and ${\bf B}_{i}$ are false for $1 \leq i \leq r-1$
and ${\bf B}_{r}$ is false.

Since the condition ${\bf B}_{i}$ is false for $i \leq r-1$
we have that $|C_{i}| = \kappa(|C_{i}|)$, that is,
$|C_{i}|$ is a part of the partition $\lambda$ for $i \leq r-1$.
Also ${\bf B}_{r}$ is false implies that $|C_{r}| = m$.
Hence
$\{|C_{1}|, \ldots, |C_{r-1}|, \underline{|C_{r}|}\}$
is the pointed partition $\{\lambda,\underline{m}\}$.
Since the condition ${\bf A}_{i}$ is false,
we have $\max(C_{i}) > \min(C_{i+1})$
or $|C_{i}| > \kappa(|C_{i+1}|) = |C_{i+1}|$
for $1 \leq i \leq r-2$.
Finally, ${\bf A}_{r-1}$ is false, implying
$\max(C_{r-1}) > \min(C_{r})$.

For the unmatched cell $(C_{1}, \ldots, C_{r})$,
let $\alpha$ be the permutation obtained by
writing each block of the critical cell in increasing order.
Furthermore, let $\dv$ be the descent composition of the
permutation~$\alpha$.
Observe that the composition $\dv$ belongs to the set
$V(\lambda, \underline{m})$
since an entry in $\dv$ is a sum of distinct parts of $\lambda$.
Furthermore the composition
$(|C_{1}|, \ldots, |C_{r-1}|, |C_{r}|)$
is the composition $\epsilon(\dv)$.
Hence the unmatched cell is
$\sigma(\alpha,\epsilon(\dv))$.
\end{proof}

\begin{example}
{\rm
Consider the pointed knapsack partition
$\{\lambda,\underline{m}\} = \{2,1,\underline{1}\}$
whose associated complex~$\Lambda_{\{2,1,\underline{1}\}}$
is shown in Figure~\ref{figure_Lambda}.
Note that
$V(\lambda,\underline{m})
   =
\{(1,2,1), (2,1,1), (3,1)\}$.
The critical cells of the complex
$\Lambda_{\{2,1,\underline{1}\}}$ are as follows:
$$ \begin{array}{c | c | c | c | l}
\dv & \beta(\dv) & \epsilon(\dv) & W(\dv) &
\text{critical cells} \\ \hline
(1,2,1) & 5 & (1,2,1) & \{(1,2,1)\} &
\text{2-14-3, 3-14-2, 3-24-1, 4-13-2, 4-23-1} \\
(2,1,1) & 3 & (2,1,1) & \{(2,1,1)\} &
\text{14-3-2, 24-3-1, 34-2-1} \\
(3,1)   & 3 & (2,1,1) & \{(1,2,1), (2,1,1)\} &
\text{12-4-3, 13-4-2, 23-4-1}
\end{array} $$
Note that
$\Lambda_{\{2,1,\underline{1}\}}$
is homotopy equivalent to a wedge of $11$ circles.
The notion $W(\dv)$ will be
defined in the beginning of the next section.
}
\label{example_2_1_1}
\end{example}

\begin{proof}[Proof of Theorem~\ref{theorem_knapsack}]
By Lemma~\ref{lemma_Morse}
the matching presented is a Morse matching
and Lemma~\ref{lemma_cells}
describes the critical cells, proving the theorem.
\end{proof}

Now by the same reasoning as in Section~\ref{section_Quillen},
that is, using the forgetful map $\phi$ and Quillen's Fiber Lemma,
we obtain a homotopy equivalence between
the order complex of pointed partitions
$\Pi^{\bullet}_{\{\lambda,\underline{m}\}} - \{\ho\}$
and the
simplicial complex of ordered set partitions
$\Lambda_{\{\lambda,\underline{m}\}}$.
Since the proof follows the same outline, it is omitted.
\begin{theorem}
The order complex
$\Delta\left(\Pi^{\bullet}_{\{\lambda,\underline{m}\}} - \{\ho\}\right)$
is homotopy equivalent to
the barycentric subdivision~$\sd(\Lambda_{\{\lambda,\underline{m}\}})$
and hence the simplicial complex
$\Lambda_{\{\lambda,\underline{m}\}}$.
\label{theorem_pointed_partitions_2}
\end{theorem}
As a corollary we obtain the
M\"obius function of the poset
$\Pi^{\bullet}_{\{\lambda,\underline{m}\}} \cup \{\hz\}$;
see~\cite{Ehrenborg_Readdy_I}.
\begin{corollary}[Ehrenborg--Readdy]
The M\"obius function of the poset
$\Pi^{\bullet}_{\{\lambda,\underline{m}\}} \cup \{\hz\}$
is given by
$$ \mu\left(\Pi^{\bullet}_{\{\lambda,\underline{m}\}} \cup \{\hz\}\right)
      =
   (-1)^{k}
      \cdot
   \sum_{\dv \in V(\lambda, \underline{m})} \beta(\dv) . $$
\label{corollary_Ehrenborg_Readdy}
\end{corollary}

\section{Cycles in the complex $\Lambda_{\{\lambda,\underline{m}\}}$}
\label{section_knapsack_cycles}

Observe that when the pointed part $m$ is equal
to zero, the complex $\Lambda_{\{\lambda,\underline{m}\}}$
is contractible. Hence we
tacitly assume that $m$ is positive
in this and the next section.

For a pointed knapsack partition
$\{ \lambda, \underline{m} \}$ of $n$
and $\dv \in V(\lambda, \underline{m})$,
let $W(\dv)$ be the set
$$ W(\dv)
      =
  \{\cv \in V(\lambda, \underline{m}):
\cv \leq \dv, \type(\cv)=\{\lambda,\underline{m}\}\} . $$
Especially we have $\epsilon(\dv) \in W(\dv)$.
For the case when
the pointed knapsack partition is
$\{2,1,\underline{1}\}$, see Example~\ref{example_2_1_1}.

For $\alpha$ a permutation in the symmetric group~$\SSSS_{n}$
and $\dv$ a composition of $V(\lambda,\underline{m})$,
define the subcomplex~$\Sigma_{\alpha, \dv}$
of the complex $\Lambda_{\{\lambda,\underline{m}\}}$
to be the simplicial complex whose facets are given by
         $$\{\sigma(\alpha \circ \gamma, \cv):
            \cv \in W(\dv), \:
            \gamma \in \SSSS^{c}_{\dv} \}.$$
This means the types of facets in $\Sigma_{\alpha, \dv}$
belong to the set $W(\dv)$.

For $\dv=(d_{1},\ldots,d_{r})$,
if $d_{i}$ splits into $t_{i}$ parts in a composition in $W(\dv)$,
then the group~$\SSSS_{t_{1}} \timesdots \SSSS_{t_{r}}$ acts on $W(\dv)$
by permuting the $t_i$ parts $d_{i}$ splits into.
Given $\cv \in W(\dv)$
there exists a permutation $\rho \in \SSSS_{t_{1}} \timesdots \SSSS_{t_{r}}$
so $\rho(\epsilon(\dv))=\cv$.
Define the sign of $\cv$,
that is, $(-1)^{\cv}$ to be the sign $(-1)^{\rho}$.
Especially, we have $(-1)^{\epsilon(\dv)}=1$.

Similar to Lemma~\ref{lemma_sphere_I} we have the next result.
\begin{lemma}
The subcomplex $\Sigma_{\alpha, \dv}$ is
isomorphic to the join
of the duals of the permutahedra
$$ P_{|K_{1}|} *  \cdots * P_{|K_{n-r+1}|}
    *
   P_{t_{1}} *  \cdots * P_{t_{r}}    $$
and hence it is a sphere of dimension $k-1$,
where $k$ is the number of parts in the partition $\lambda$.
\label{lemma_sphere_II}
\end{lemma}
Since this lemma is not
necessary for the remainder of this paper,
the proof is omitted. However, let us verify the
dimension of the sphere.
First the dimension of the $(n-r)$-fold join
$P_{|K_{1}|} *  \cdots * P_{|K_{n-r+1}|}$
is $(|K_{1}| - 2) + \cdots + (|K_{n-r+1}|-2) + (n-r)
=
n - 2 \cdot (n-r+1) + n-r = r-2$.
Similarly, the dimension of
$P_{t_{1}} *  \cdots * P_{t_{r}}$
is given by
$(t_{1}-2) + \cdots + (t_{r}-2) + (r-1)
  =
(k+1) - 2 \cdot r + r-1 = k-r$.
Thus the dimension of the sphere in the lemma
is $(r-2) + (k-r) + 1 = k-1$.

Define the element $g_{\alpha,\dv}$
in the chain group $C_{k-1}(\Lambda_{\{\lambda,\underline{m}\}})$ to be
$$ g_{\alpha,\dv}
      =
   \sum_{\gamma \in \SSSS^{c}_{\dv}}
   \sum_{\cv \in W(\dv)}
        (-1)^{\gamma} \cdot (-1)^{\cv} \cdot
        \sigma(\alpha \circ \gamma, \cv).$$
When $\alpha$ has descent composition $\dv$,
note the critical cell $\sigma(\alpha, \epsilon(\dv))$
has sign $1$ in $g_{\alpha,\dv}$.

\begin{lemma}
For $\alpha \in \SSSS_{n}$, the element $g_{\alpha,\dv}$
in the chain group $C_{k-1}(\Lambda_{\{\lambda,\underline{m}\}})$
belongs to the kernel of the boundary map
and hence the homology group
$\HH_{k-1}(\Lambda_{\{\lambda,\underline{m}\}})$.
\end{lemma}
\begin{proof}
Apply the boundary map $\partial$ to $g_{\alpha,\dv}$
in $C_{k-1}(\Lambda_{\{\lambda,\underline{m}\}})$
and exchange the order of the three sums
to obtain
$$
\partial(g_{\alpha,\dv})
=
  \sum_{i=1}^{k}
  (-1)^{i-1} \cdot
  \sum_{\gamma \in \SSSS^{c}_{\dv}}
  \sum_{\cv \in W(\dv)}
      (-1)^\gamma \cdot
      (-1)^{\cv} \cdot
        (\ldots, \{ \ldots, \alpha_{\gamma_{c_{1}+\cdots+c_{i}}},
        \alpha_{\gamma_{c_{1}+\cdots+c_{i}+1}}, \ldots\}, \ldots) .
$$
If $\sum_{s=1}^{i} c_{s}=\sum_{s=1}^{j} d_{s}$ for some $j$,
the term corresponding to $\gamma$ cancels
with
the term corresponding to
$\gamma \circ (c_{1}+\cdots+c_{i}, c_{1}+\cdots+c_{i}+1)$.
If there does not exist such an index $j$,
we find the smallest integer~$\ell$ such that
$\sum_{s=1}^{i} c_{s} < \sum_{s=1}^{\ell} d_{s}$.
Note that $d_{\ell}$ splits into $t_{\ell} \geq 2$ parts.
Now consider the composition
$\cvp = (\ldots, c_{i-1},c_{i+1},c_{i},c_{i+2}, \ldots)$
where we switched the $i$th and the $(i+1)$st parts of $\cv$.
Note that $\cvp$ also belongs to $W(\dv)$ and the sign
differs from the sign of $\cv$, that is,
$(-1)^{\cvp} = -(-1)^{\cv}$.
Then the term corresponding to $\cv$ cancels with
the term corresponding to $\cvp$.
Hence, the sum vanishes.
\end{proof}

\begin{lemma}
If the critical cell $\sigma(\alpha^{\prime},\epsilon(\dvp))$
belongs to the cycle $g_{\alpha,\dv}$,
then we have the inequality $\alpha^{\prime} \leq \alpha$
in the weak Bruhat order.
Furthermore, if
the two permutations $\alpha$ and $\alpha^{\prime}$
are equal, then the two compositions~$\dv$ and $\dvp$
are equal.
\label{lemma_knapsack_triangular}
\end{lemma}
\begin{proof}
Since $\sigma(\alpha^{\prime},\epsilon(\dvp))$ belongs to $g_{\alpha,\dv}$,
we have $\sigma(\alpha^{\prime},\epsilon(\dvp))=\sigma(\alpha \circ \gamma,\cv)$
for some $\gamma \in \mathfrak{S}'_{\dv}$ and $\cv \in W(\dv)$.
Note that $\epsilon(\dvp) = \cv$.
Hence a permutation $\alpha^{\prime}$
corresponding to the critical cell $\sigma(\alpha^{\prime},\epsilon(\dvp))$
is obtained
by writing each block of $\sigma(\alpha \circ \gamma,\cv)$ in increasing
order
and this implies $\alpha^{\prime} \leq \alpha \circ \gamma$.
Furthermore, $\alpha \circ \gamma \leq \alpha$ in the weak Bruhat order by
Lemma~\ref{lemma_Bruhat}.
Hence, we have $\alpha^{\prime} \leq \alpha$.

Finally, if the two permutations $\alpha$ and $\alpha^{\prime}$
are equal, we have that
$\dv = \Des(\alpha) = \Des(\alpha^{\prime}) = \dvp$.
\end{proof}

Using Lemma~\ref{lemma_knapsack_triangular}
we have the next result.
Observe that the proof is similar to that of Theorem~\ref{theorem_homology}
and hence omitted.

\begin{theorem}
For a pointed knapsack partition
$\{ \lambda, \underline{m} \}$ of $n$,
the cycles $g_{\alpha,\dv}$ where
the composition~$\dv$ ranges over the set $V(\lambda, \underline{m})$
and
$\alpha$ ranges over
all permutations with descent composition~$\dv$,
form a basis for the homology group
$\HH_{k-1}(\Lambda_{\{\lambda,\underline{m}\}})$.
\end{theorem}

\section{The group action on the top homology}
\label{section_knapsack_group}

By the equivariant homology version of
Quillen's Fiber Lemma,
Theorem~\ref{theorem_equivariant_Quillen},
we have the following isomorphism.
\begin{theorem}
The two homology groups
$\HH_{k-1}\left(\Delta\left(\Pi^{\bullet}_{\{\lambda,\underline{m}\}} - \{\ho\}\right)\right)$
and
$\HH_{k-1}\left(\Lambda_{\{\lambda,\underline{m}\}}\right)$
are isomorphic as $\SSSS_{n}$-modules.
\label{theorem_knapsack_homology_modules}
\end{theorem}
Hence it remains to make the connection between
the action on ordered set partitions and Specht modules.
\begin{theorem}
The direct sum of Specht modules
$$ \bigoplus_{\dv \in V(\lambda, \underline{m})} S^{B(\dv)} $$
is isomorphic to the top homology group
$\HH_{k-1}\left(\Lambda_{\{\lambda,\underline{m}\}}\right)$
as $\SSSS_{n}$-modules.
\end{theorem}
\begin{proof}
We define a homomorphism
$$ \Psi :
   \bigoplus_{\dv \in V(\lambda, \underline{m})} M^{B(\dv)}
   \longrightarrow
   C_{k-1}\left(\Lambda_{\{\lambda,\underline{m}\}}\right)  ,  $$
by sending a tabloid $s$ of shape $B(\dv)$ to
the element $\sum_{\cv \in W(\dv)} (-1)^{\cv} \sigma(\alpha, \cv)$
in the chain group, where $\alpha$ is the permutation
obtained from the tabloid $s$
by reading the elements in each row in increasing order.
Observe that $\Psi$ is a $\SSSS_{n}$-module homomorphism.

Now consider the restriction of the homomorphism $\Psi$ to
the direct sum of Specht modules.
Note that
the group~$\SSSS_{\dv}^{c}$
is the column stabilizer of the border strip $B(\dv)$.
Let $t$ be a tableau and $\alpha$ its associated 
permutation.
The homomorphism $\Psi$ applied to the polytabloid ${\bf e}_{t}$
(see reference~\cite[Definition~2.3.2]{Sagan})
is
as follows:
$$    \Psi({\bf e}_{t})
    =
       \sum_{\gamma \in \SSSS_{\dv}^{c}}
          (-1)^{\gamma}
            \cdot
      \sum_{\cv \in W(\dv)}
          (-1)^{\cv} \cdot \sigma(\alpha \circ \gamma, \cv)
   =
      g_{\alpha, \dv}  ,  $$
which belongs to the kernel of the boundary map.
Hence $\Psi$ maps the directed sum of the Specht modules
to the homology group
$\HH_{k-1}\left(\Lambda_{\{\lambda,\underline{m}\}}\right)$.

Since $g_{\alpha,\dv}$ lies in the image of the restriction of $\Psi$
and the elements $g_{\alpha,\dv}$ span the homology group,
the restriction is surjective. Furthermore,
since the two $\SSSS_{n}$-modules have the same dimension,
we conclude that they are isomorphic.
\end{proof}

Hence we conclude
\begin{theorem}
The top homology group
$\HH_{k-1}\left(\Delta\left(\Pi^{\bullet}_{\{\lambda,\underline{m}\}} - \{\ho\}\right)\right)$
and
the direct sum of Specht modules
$$ \bigoplus_{\dv \in V(\lambda, \underline{m})} S^{B(\dv)} $$
are isomorphic as $\SSSS_{n}$-modules.
\label{theorem_Specht_2}
\end{theorem}

\section{Concluding remarks}

We have not dealt with the question
whether the poset $\Pi^{\bullet}_{\cv}$
is $EL$-shellable. Recall that Wachs
proved that the $d$-divisible partition lattice
$\Pi^{d}_{n} \cup \{\hz\}$
has an $EL$-labeling.
Ehrenborg and Readdy
gave an extension of this labeling
to prove that
$\Pi^{\bullet}_{(d, \ldots, d, m)}$
is $EL$-shellable~\cite{Ehrenborg_Readdy_II}.
Woodroofe~\cite{Woodroofe}
has shown that the order complex
$\Delta(\Pi^{d}_{n} - \{\ho\})$ has a convex ear decomposition.
This is not true in general for
$\Delta(\Pi^{\bullet}_{\cv} - \{\ho\})$.
See for instance
$\Delta(\Pi^{\bullet}_{(1,1,1)} - \{\ho\})$
in Figure~\ref{figure_pointed_hexagon}.

Can the order complex
$\Delta(\Pi^{\bullet}_{\cv} - \{\ho\})$
be shown to be shellable, using the fact that
the complex $\Delta_{\cv}$ is shellable?
That is, can a shelling of
$\Delta_{\cv}$ be lifted, similar to using
Quillen's Fiber Lemma?
Would this shelling relationship
also extend to
the two complexes
$\Lambda_{\{\lambda,\underline{m}\}}$
and
$\Delta\left(\Pi_{\{\lambda,\underline{m}\}} - \{\ho\}\right)$?

Kozlov~\cite{Kozlov_paper} introduced the notion of
a poset to be $EC$-shellable.
He showed
that the filter $\Pi_{\lambda}$ in the partition lattice
generated by a knapsack partition $\lambda$ is
$EC$-shellable~\cite[Theorem~4.1]{Kozlov_paper}.
Is it possible to
show that $\Pi^{\bullet}_{\cv}$
and $\Pi^{\bullet}_{\{\lambda,\underline{m}\}}$
are $EC$-shellable?
Furthermore, Kozlov computes the M\"obius function
of $\Pi_{\lambda}$.
See~\cite[Corollary~8.5]{Kozlov_paper}.
Can his
answer be expressed in terms of permutation statitics,
such as the descent set statistic?

Can these techniques be used for studying
other subposets of the partition lattice?
One such subposet is the odd partition lattice,
that is, the collection of all partitions where
each block size is odd. More generally,
what can be said about the case when all
the block sizes are congruent to $r$ modulo $d$?
These posets have been studied
in~\cite{Calderbank_Hanlon_Robinson}
and~\cite{Wachs_II}.
Moreover, what can be said about the poset
$\Pi^{\bullet}_{\{\lambda,\underline{m}\}}$
when $\{\lambda,\underline{m}\}$ is not
a pointed knapsack partition?

Another analogue of the partition lattice
is the Dowling lattice. Subposets
of the Dowling lattice have been
studied in~\cite{Ehrenborg_Readdy_II}
and~\cite{Gottlieb_I,Gottlieb_Wachs}.
Here the first question to ask is:
what is the right analogue of ordered set partitions?

\begin{figure}[t]
\setlength{\unitlength}{0.5mm}
\begin{center}
\begin{picture}(120,104)(0,0)

\put(0,2){\circle*{3}}
\put(60,104){\circle*{3}}
\put(120,2){\circle*{3}}

\put(30,54){\circle*{3}}
\put(90,54){\circle*{3}}
\put(60,2){\circle*{3}}

\put(30,17.4){\circle*{3}}
\put(60,69.4){\circle*{3}}
\put(90,17.4){\circle*{3}}

\thicklines

\put(30,17.4){\line(2,-1){30}}
\put(30,17.4){\line(-2,-1){30}}
\put(30,17.4){\line(0,1){36.6}}

\put(90,17.4){\line(2,-1){30}}
\put(90,17.4){\line(-2,-1){30}}
\put(90,17.4){\line(0,1){36.6}}

\put(60,69.4){\line(2,-1){30}}
\put(60,69.4){\line(-2,-1){30}}
\put(60,69.4){\line(0,1){34.6}}

\put(-11,-5){$13|\underline{2}$}
\put(121,-4){$12|\underline{3}$}
\put(56,107){$23|\underline{1}$}

\put(56,-5){$1|\underline{23}$}
\put(92,51){$2|\underline{13}$}
\put(13,51){$3|\underline{12}$}

\put(11,17.4){$1|3|\underline{2}$}
\put(92,17.4){$1|2|\underline{3}$}
\put(63,69.4){$2|3|\underline{1}$}

\end{picture}
\end{center}
\caption{The order complex of the poset
$\Pi^{\bullet}_{(1,1,1)} - \{\ho\}$.}
\label{figure_pointed_hexagon}
\end{figure}
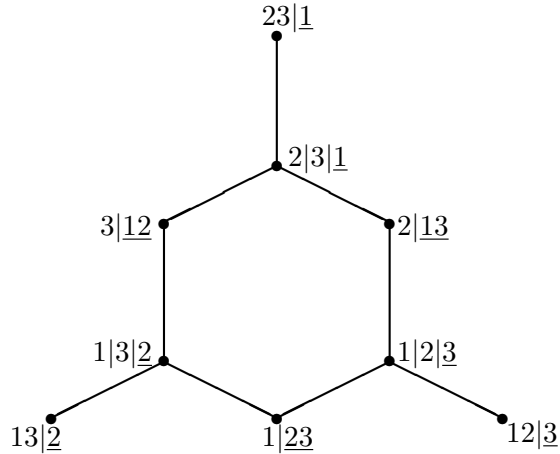

Wachs gave a basis for the top homology of
the order complex of the $d$-divisible partition
lattice~\cite[Section~2]{Wachs}.
Each cycle in her basis is
the barycentric subdivision of the boundary of a cube.
We can similarly describe a basis for
the order complex of $\Pi^{\bullet}_{\cv}$.
The major difference is that
the cycles are the barycentric
subdivision of a different polytope
depending on the composition $\cv$.
In order to describe this polytope,
recall that the $n$-dimensional root polytope $R_{n}$
(of type $A$) is defined
as the intersection of the $(n+1)$-dimensional crosspolytope
$\conv(\{\pm2 \ev_{i}\}_{1 \leq i \leq n+1})$ and
the hyperplane $x_{1} + x_{2} + \cdots + x_{n+1} = 0$.
Equivalently, the root polytope $R_{n}$ can be defined as
the convex hull of the set
$\{\ev_{i} - \ev_{j}\}_{1 \leq i, j \leq n+1}$.
Lastly, let $S_{n}$ denote the $n$-dimensional simplex.

We state the following theorem without proof.
\begin{theorem}
Given a composition $\cv$,
there is a basis for
$\HH_{k-2}(\Delta(\Pi^{\bullet}_{\cv} - \{\ho\}))$
where each basis element is the barycentric subdivision
of the boundary of the Cartesian product:
$$
\begin{array}{c l}
R_{|K_{1}|-1} \times R_{|K_{2}|-1} \timesdots
R_{|K_{n-k}|-1} \times S_{|K_{n-k+1}|-1}
& \text{ if } c_{1} \neq 1, \\
S_{|K_{1}|-1} \times R_{|K_{2}|-1} \timesdots
R_{|K_{n-k}|-1} \times S_{|K_{n-k+1}|-1}
& \text{ if } c_{1} = 1.
\end{array}
$$
Note that when all the parts of the composition $\cv$
are greater than $1$, this polytope reduces to
the $(k-1)$-dimensional cube.
\end{theorem}

In a recent preprint Miller~\cite{Miller} has extended the definition
of the partition poset $\Pi^{\bullet}_{\cv}$ from type $A$
to all real reflection groups and
the complex reflection groups known as Shephard groups.
Can his techniques also extend
our results for the filter
$\Pi^{\bullet}_{\{\lambda, \underline{m}\}}$
where
$\{\lambda, \underline{m}\}$
is a knapsack partition?

\section*{Acknowledgements}

The authors thank Serge Ochanine for many helpful discussions
and Margaret Readdy, Michelle Wachs and the two referees for their comments on an earlier draft of this paper.
Both authors are partially supported by
National Science Foundation grant DMS-0902063.
The first author also thanks the Institute for Advanced
Study and is also partially supported by
National Science Foundation grants
DMS-0835373 and CCF-0832797.

\newcommand{\journal}[6]{{\sc #1,} #2, {\it #3} {\bf #4} (#5), #6.}
\newcommand{\book}[4]{{\sc #1,} ``#2,'' #3, #4.}
\newcommand{\bookf}[5]{{\sc #1,} ``#2,'' #3, #4, #5.}
\newcommand{\books}[6]{{\sc #1,} ``#2,'' #3, #4, #5, #6.}
\newcommand{\collection}[6]{{\sc #1,}  #2, #3, in {\it #4}, #5, #6.}
\newcommand{\thesis}[4]{{\sc #1,} ``#2,'' Doctoral dissertation, #3, #4.}
\newcommand{\springer}[4]{{\sc #1,} ``#2,'' Lecture Notes in Math., Vol.\ #3,
Springer-Verlag, Berlin, #4.}
\newcommand{\preprint}[3]{{\sc #1,} #2, preprint #3.}
\newcommand{\preparation}[2]{{\sc #1,} #2, in preparation.}
\newcommand{\appear}[3]{{\sc #1,} #2, to appear in {\it #3}}
\newcommand{\submitted}[3]{{\sc #1,} #2, submitted to {\it #3}}
\newcommand{\JCTA}{J.\ Combin.\ Theory Ser.\ A}
\newcommand{\AdvancesinMathematics}{Adv.\ Math.}
\newcommand{\JournalofAlgebraicCombinatorics}{J.\ Algebraic Combin.}

\bigskip
\noindent
{\em R.\ Ehrenborg and J.\ Jung,
    Department of Mathematics,
    University of Kentucky,
    Lexington, KY~40506},
{\verb+{jrge,jjung}@ms.uky.edu+}.

\end{document}